\documentclass[a4paper,11pt]{article}

\usepackage[T1]{fontenc}
\usepackage{lmodern}

\usepackage{dsfont}

\usepackage[utf8]{inputenc}
\usepackage{amsmath}
\usepackage{amsthm}
\usepackage{amssymb}
\usepackage{mathrsfs}
\usepackage{graphicx}
\usepackage[all]{xy}
\usepackage{hyperref}
\usepackage[usenames, dvipsnames]{xcolor}
\usepackage{soul}
\usepackage{thm-restate}
\usepackage{tikz-cd}
\usepackage{pdfpages}

\usepackage{stmaryrd}
\usepackage{caption}
\usepackage{subcaption}
\usepackage{abstract} 
\usepackage[shortlabels]{enumitem}

\usepackage{tikz} 
\usetikzlibrary{decorations.pathreplacing,decorations.markings}

 \tikzset{mid arrow/.style={postaction={decorate,decoration={
			markings,
			mark=at position .5 with {\arrow[#1]{stealth}}
}}}}

\newtheorem{thm}{Theorem}[section]
\newtheorem{cor}[thm]{Corollary}
\newtheorem{claim}[thm]{Claim}
\newtheorem{fact}[thm]{Fact}

\newtheorem{lemma}[thm]{Lemma}
\newtheorem{prop}[thm]{Proposition}

\theoremstyle{definition}

\newtheorem{ex}[thm]{Example}
\newtheorem{remark}[thm]{Remark}
\newtheorem{question}[thm]{Question}

\newtheorem{conj}[thm]{Conjecture}

\newtheorem*{definition*}{Definition}

\DeclareMathOperator{\N}{\mathbb{N}}
\DeclareMathOperator{\stab}{\mathrm{stab}}
\DeclareMathOperator{\Mod}{\mathrm{Mod}}

\DeclareMathOperator{\Z}{\mathbb{Z}}
\DeclareMathOperator{\id}{{id}}
\DeclareMathOperator{\Fr}{\mathrm{Fr}}

\renewcommand{\l}{\ell}

\newlist{cas}{enumerate}{1}
\setlist[cas,1]{label=Case \arabic*}

\def\rquotient#1#2{%
	\makeatletter
	\raise.3ex\hbox{$#1$}/\lower.3ex\hbox{$#2$}%
	\makeatother
}	

\makeatletter
\newcommand{\subjclass}[2][2010]{%
	\let\@oldtitle\@title%
	\gdef\@title{\@oldtitle\footnotetext{#1 \emph{Mathematics subject classification.} #2}}%
}
\newcommand{\keywords}[1]{%
	\let\@@oldtitle\@title%
	\gdef\@title{\@@oldtitle\footnotetext{\emph{Key words and phrases.} #1.}}%
}
\makeatother

\newcommand{\Address}{{% additional braces for segregating \footnotesize
		\bigskip
		\small
		
		\textsc{Institut Montpellierain Alexander Grothendieck, 499-554 Rue du Truel, 34090 Montpellier, France.}\par\nopagebreak
		\textit{E-mail address}: \texttt{anthony.genevois@umontpellier.fr}
		\medskip
		
		\textsc{Laboratoire de mathématiques d’Orsay, Université Paris-Saclay, 91405, Orsay, France}\par\nopagebreak
		\textit{E-mail address}: \texttt{anne.lonjou@universite-paris-saclay.fr}
		\medskip
		
		\textsc{Departement Mathematik, ETH Zurich, Rämistrasse 101, CH-8092 Zürich, Swit\-zer\-land.}\par\nopagebreak
		\textit{E-mail address}: \texttt{christian.urech@math.ethz.ch}
		\medskip
		
}}

\title{Asymptotically rigid mapping class groups III: Presentations and isomorphisms}
\date{\today}

\author{Anthony Genevois, Anne Lonjou, and Christian Urech}


\begin{document}

	\maketitle

	\begin{abstract}
		\noindent This article is dedicated to the computation of an explicit presentation of some asymptotically rigid mapping class groups, namely the braided Higman-Thompson groups. To do so, we use the action of these groups on the spine complex, a simply connected cube complex constructed by the authors in a previous work. In particular, this allows to compute the abelianisations of these groups. With these new algebraic invariants we can handle many new cases of the isomorphism problem for asymptotically rigid mapping class groups of trees.
	\end{abstract}

	\tableofcontents

	\section{Introduction}

Given a surface of infinite type $\Sigma$ endowed with a fixed cellulation, referred as a \emph{rigid structure}, the \emph{asymptotically rigid mapping class group} $\mathfrak{mod}(\Sigma)$ is the subgroup of the big mapping class group $\mathrm{mod}(\Sigma)$ given by the homeomorphisms $\Sigma \to \Sigma$ that map cells to cells with only finitely many exceptions. Often, despite the fact that $\mathrm{mod}(\Sigma)$ is not even countable, its asymptotically rigid subgroup $\mathfrak{mod}(\Sigma)$ turns out to satisfy good finiteness properties, providing an interesting source of finitely generated groups. See for instance \cite{Degenhardt, Funar_Kapoudjian_UniversalMCG, Funar_Houghton, Funar-Kapoudjian_brT_finitely_presented, Aramoyana_Funar_asymptotic_MCG, GLU_finiteness, GLU_Chambord, SurfaceHoughton}.

\medskip\noindent In this article, we pursue our study of asymptotically rigid mapping class groups of planar surfaces initiated in \cite{GLU_finiteness}. Given a locally finite tree $A$ properly embedded in the plane, consider the planar surface $\mathscr{S}(A)$ given by a small closed tubular neighbourhood of $A$; and let $\mathscr{S}^\# (A)$ be the surface obtained from $\mathscr{S}(A)$ by adding a puncture at each vertex of $A$. A rigid structure can be naturally defined on $\mathscr{S}^\#(A)$ by adding arcs transverse to the edges of $A$. We denote by $\mathfrak{mod}(A)$ the asymptotically rigid mapping class group of $\mathscr{S}^\#(A)$ endowed with its rigid structure (see Figure~\ref{fig:homeo} for an example of an element of $\mathfrak{mod}(A)$ when $A$ is the infinite $3$-regular tree). The main question we are interested in is:

\begin{question}\label{question:Intro}
Given two trees $A_1$ and $A_2$, when are $\mathfrak{mod}(A_1)$ and $\mathfrak{mod}(A_2)$ isomorphic?
\end{question}

\noindent Of course, if there exists a quasi-isomorphism $A_1 \to A_2$, i.e.\ a bijection on the vertex-sets $V(A_1) \to V(A_2)$ that preserves adjacency and non-adjacency for all but finitely many pairs of vertices, then there exists an asymptotically rigid homeomorphism $\mathscr{S}^\# (A_1) \to \mathscr{S}^\#(A_2)$ that induces a group isomorphism $\mathfrak{mod}(A_1) \to \mathfrak{mod}(A_2)$. But the converse does not hold. In fact, there exist many trees $A$ with so few symmetries that every asymptotically rigid homeomorphism $\mathscr{S}^\# (A_1) \to \mathscr{S}^\#(A_2)$ must be compactly supported, which implies that $\mathfrak{mod}(A)$ reduces to $B_\infty$ (i.e.\ the group of finitely supported braids on countably many strands).

\begin{figure}
	\begin{center}
			\includegraphics[scale=0.2]{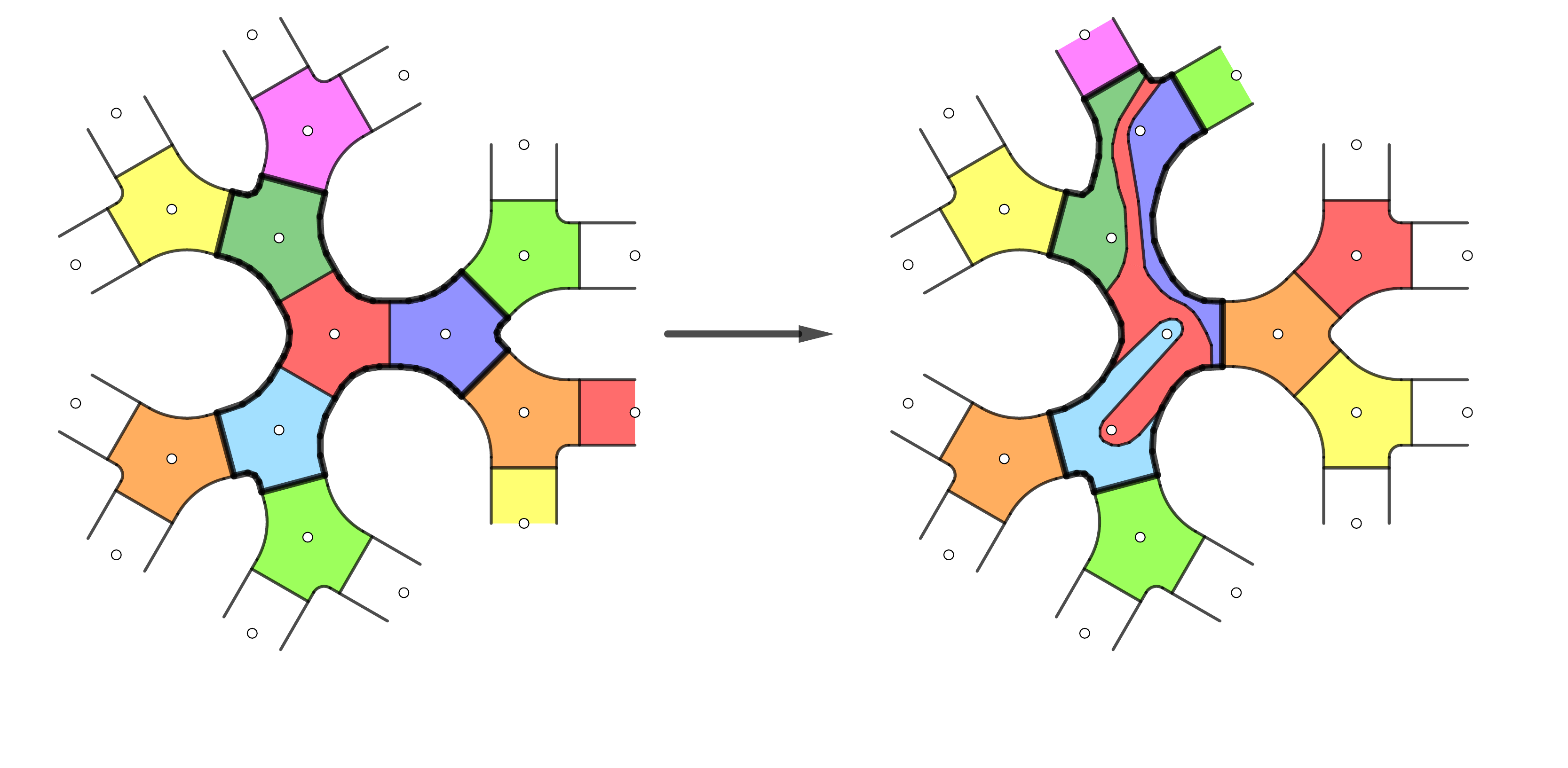}
			\caption{An element of $\mathrm{mod}(\mathscr{S}^\#(A_{2,3}))$ \label{fig:homeo}}
		\end{center}
\end{figure}

\medskip\noindent Question~\ref{question:Intro} seems out of reach in full generality, so in this article we restrict ourselves to a specific family of asymptotically rigid mapping class groups, namely the \emph{braided Higman-Thompson groups} $\mathrm{br}T_{n,m}:= \mathfrak{mod}(A_{n,m})$, where $A_{n,m}$ is the rooted tree whose root has degree $m$ and all of whose other vertices have degree $n+1$ (see Figure \ref{fig:homeo} for an example of an element of $\mathrm{br}T_{2,3}$). The terminology comes from the observation that the forgetful map $\mathrm{mod}(\mathscr{S}^\#(A_{n,m})) \to \mathrm{mod}(\mathscr{S}(A_{n,m}))$ induces a short exact sequence
$$1 \to B_\infty \to \mathrm{br}T_{n,m} \to T_{n,m} \to 1$$
where $T_{n,m}$ is the Higman-Thompson group corresponding to $A_{n,m}$. Interestingly, despite the fact that there exist non-trivial isomorphisms between certain Higman-Thomp\-son groups, their braided versions seem to be more rigid. For instance, $T_{n,m}$ and $T_{n,m+k(n-1)}$ are isomorphic for every $k\in\Z$, but $\mathrm{br}T_{n,m}$ and $\mathrm{br}T_{n,m+k(n-1)}$ turn out not to be isomorphic for $k\neq 0 , -1$ since they do not have the same torsion (according to \cite{GLU_finiteness}). 

\begin{conj}\label{conj}
For all $n,m,r,s \geq 2$, the groups $\mathrm{br}(T_{n,m})$ and $\mathrm{br}(T_{r,s})$ are isomorphic if and only if $(n,m)=(r,s)$. 
\end{conj}

\noindent In this article, we exploit various algebraic invariants in order to verify the conjecture in many cases:

\begin{thm}\label{thm:BigIntro}
Let $n,m,r,s \geq 2$ be integers. If $\mathrm{br}T_{n,m}$ and $\mathrm{br}T_{r,s}$ are isomorphic, then $(r,s)=(n,m)$ or $2\leq m\leq \frac{n-1}{2}$ and $(r,s )=(n,n-1-m)$.
\end{thm}

\noindent Our strategy is twofold. First, for $n,m\geq 2$, we deduce from the action of $\mathrm{br}(T_{n,m})$ on the contractible cube complex constructed in \cite{GLU_finiteness} a presentation of the group using Brown's method. Note that for the special case $(n,m)=(2,3)$ an explicit presentation has been computed in \cite{Funar-Kapoudjian_brT_finitely_presented}. In Theorem ~\ref{thm_presentation_BHT} we give a presentation for all $n,m\geq2$. 
For instance, for all $n \geq 2$ and $m \geq 4$, the group $\mathrm{br}(T_{n,m})$ admits a presentation with generators $r_0, \dots, r_4$ and $\tau_1, \ldots, \tau_4$, and with the relations
\begin{itemize}
	\item \emph{the braids relations: } 
		\begin{enumerate}
			\item $\tau_i\tau_j\tau_i=\tau_j\tau_i\tau_j$, for any $1\leq i<j\leq 4$,
			\item $\tau_i\tau_j\tau_s\tau_i=\tau_j\tau_s\tau_i\tau_j=\tau_s\tau_i\tau_j\tau_s$, for any $1\leq i<j<s\leq 4$,
		\end{enumerate}
	\item \emph{the commutation relations:} $r_{k}\tau_i=\tau_ir_{k}$ for $1\leq i\leq k\leq 4$,
	\item \emph{the rotation relations:} $r_{k}^{m+k(n-1)}=(\tau_k\tau_{k-1}\dots\tau_1)^{-(k+1)}$ for $0\leq k\leq 4$,
	\item \emph{the square relations}: for $1\leq i\leq 3$ and $1\leq j_i \leq \left\lceil\frac{m+(n-1)(i-1)-1}{2}\right\rceil$
	
		\[	r_{i-1}^{j_i}\tau_{i}^{-1}r_{i}^{-n-j_i}\tau_{i+i}\tau_{i}r_{i+1}^{j_i+n-1}r_{i}^{1-j_i}=\id .\]
\end{itemize}
We refer to Section~\ref{section:set-up} for a topological description of the generators. As an easy consequence of our calculation, the abelianisations of the braided Higman-Thompson groups can be computed. 

\begin{thm}\label{thm_abelianise}
For all $n,m \geq 2$, the abelianisation of $\mathrm{br}T_{n,m}$ is $\mathbb{Z}_m \times \mathbb{Z}_{|m-n+1|}$.
\end{thm}

\noindent We recover from this also the known abelianization of $T_{n,m}$, see Remark \ref{rmk_abel_tnr}.

\noindent Theorem \ref{thm_abelianise} provides the first algebraic invariant used in the proof of Theorem~\ref{thm:BigIntro}. Next, we show that the subgroup $B_\infty$ in $\mathrm{br}T_{n,m}$ can be characterised algebraically.

\begin{thm}\label{thm_thompson_iso}
Let $n,m \geq 2$ be integers. The subgroup $B_\infty$ of $\mathrm{br}T_{n,m}$ is the unique subgroup that is maximal (with respect to the inclusion) among the subgroups $N$ satisfying the property
\begin{itemize}
	\item[($\ast$)] $N$ is normal and $\mathrm{br}T_{n,m}/N$ does not surject onto a virtually abelian group with a kernel that has a non-trivial centre. 
\end{itemize}
\end{thm}

\noindent As a consequence, every isomorphism $\mathrm{br}T_{n,m} \to \mathrm{br}T_{r,s}$ sends $B_\infty \leq \mathrm{br}T_{n,m}$ to $B_\infty \leq \mathrm{br}T_{r,s}$, and therefore induces an isomorphism $T_{n,m} \to T_{r,s}$. From a standard application of Rubin's theorem, we know that such an isomorphism imposes that $n=r$ (see Proposition~\ref{prop:IsoThompson}). This is the second ingredient in our proof of Theorem~\ref{thm:BigIntro}.

\medskip\noindent It is worth noticing that, even though we conjecture they are not isomorphic, the groups $\mathrm{br}T_{n,m}$ and $\mathrm{br}T_{n,n-1-m}$ (with $2\leq m\leq \frac{n-1}{2}$) turn out to share many algebraic invariants. For instance, they have the same torsion, the same abelianisation. They also seem to have the same number of conjugacy classes of torsion elements. We do not know if their underlying Higman-Thompson groups are isomorphic or not. If they are not isomorphic, Theorem~\ref{thm_thompson_iso} would prove Conjecture~\ref{conj}.

\subsection*{Outline of the article}
In Section \ref{section_preliminaries}, we recall Brown's method and illustrate it with an example, we then recall the definition of the braided Higman-Thompson groups, the construction of the spine complex, and a suitable presentation of the braid group that we will need. We also show that a suitable subcomplex of the spine complex is simply in order to make the computation of the presentation easier. Section \ref{section_presentation} is the heart of the article and is dedicated to an explicit computation of the presentation of the braided Higman-Thompson groups. The most technical part conists of the computation of the relations corresponding to fundamental squares (see Section \ref{Subsection:relation_squares}). Finally, in Section \ref{section_abelianisation}, we compute the abelianisation of $\mathrm{br}T_{n,m}$ and in Section \ref{section_isom_problem} we prove the main results on the isomorphism problem.

\subsection*{Acknowledgements} The authors would like to thank Jim Belk for explaining us the known results about the isomorphism problem in the family of Higman-Thompson groups $T_{n,m}$. The second author is partially supported by MCIN /AEI /10.13039/501100011033 / FEDER through the spanish grant Proyecto PID2022-138719NA-I00. and by the french National Agency of Research (ANR) through the project GOFR ANR-22-CE40-0004. The third author was partially supported by the SNSF grant 10004735.

	\section{Preliminaries}\label{section_preliminaries}
In Subsection \ref{subsection:Brown}, we describe a specific case of Brown's method (\cite{Brown_presentation}) that will be used to compute an explicit presentation of the the braided Higman-Thompson groups $\mathrm{br}T_{n,m}$. Then we recall the construction of the braided Higman-Thompson groups in Subsection \ref{subsection:HT}, the construction of the spine cube complex made in \cite{GLU_finiteness} in Subsection \ref{subsection:spine} and a suitable presentation (for our context) of braid groups in Subsection \ref{Subsection_braids}. Finally, in Subsection \ref{subsection:spine_height}, we prove that the subcomplex of the spine complex given by vertices of height at most $5$ if $m=n=2$, or $4$ otherwise, is  simply connected. 

\subsection{Brown's method}\label{subsection:Brown}
We recall in this subsection Brown's method (\cite{Brown_presentation}) not in the full generality, but only in the framework that we need for later. We illustrate this method with an easy example at the end of this subsection. 

\medskip\noindent   Consider the action of a group $G$ on an oriented CW simply connected complex $X$ that preserves the orientation. Choose such an orientation on $X$. For any edge $e$ in $X$, according to the orientation on $X$, we denote by $o(e)$ its vertex of origin and by $t(e)$ its terminal vertex. Now, several choices have to be done.
\begin{enumerate}[wide]
\item First, we find a \emph{tree of representatives}, meaning a tree $T$ such that its set of vertices $V$ is a set of representatives of the vertices of $X$ under the action of $G$.
\item Second, we choose a set $E^+$ of representatives of edges for the action of $G$ on $X$ such that for each edge $e\in E^+$, $o(e)\in V$ and for any edge $\tilde{e}$ of $T$, $\tilde{e}\in E^+$. If $E^+$ corresponds to the set of edges of $T$ then $G$ is generated by the isotropy subgroups of the vertices of $T$: $\{G_v\}_{v\in V}$.

 	\item Last, we choose a set $F$ of representatives of $2$-cells under the action of $G$ such that any representative is based on a vertex belonging to $V$. To each element of $F$ corresponds a relation. In order to do it, we first associate to each edge $\alpha$ of $X$ (with an orientation possibly different from the one fixed) starting in a vertex of $V$ the following element $h\in G_{o(\alpha)}$. Depending on the orientation of $\alpha$ (see Figure \ref{fig:element_edge}), $h$ is chosen as follows.
 		\begin{figure}
 		\begin{subfigure}{.45\textwidth}
 			\begin{center}
 				\begin{tikzcd}				
 					{o(e)=o(\alpha) } \arrow[d,red, "e" black]\arrow[r,red,"\alpha" black]	 & {t(\alpha) }\\
 					{t(e) } \arrow[ur, dashrightarrow, bend right, red, "h" red]  \end{tikzcd}
 				\caption{The edge $\alpha$ has the same orientation in the complex and in the square.}
 				\label{fig:same_orientation}
 			\end{center}
 		\end{subfigure}
 		\hfill
 		\begin{subfigure}{.45\textwidth}
 			\begin{center}
 				\begin{tikzcd}
 					{t(e)=t(\alpha) } \arrow[r,red,"\alpha" black]	\arrow[d,red, "e" black]  & {o(\alpha) }\\
 					{o(e) } \arrow[ur, dashrightarrow, bend right, red, "h" red]  
 				\end{tikzcd}
 				\caption{The edge $\alpha$ does not have the same orientation in the complex and in the square.}
 				\label{fig:opposite_orientation}
 			\end{center}
 		\end{subfigure}	
 		\caption{Choice of an element $h\in G_{o(\alpha)}$ to the edge $\alpha$ of a square.}
 		\label{fig:element_edge}
 	\end{figure}
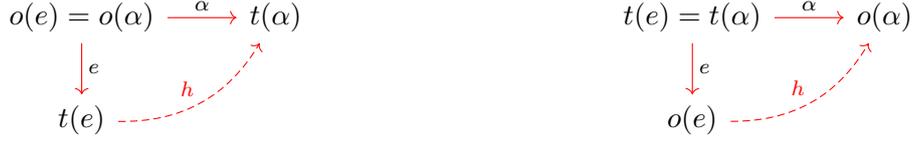
 	\begin{itemize}
 		\item If the direction of the edge $\alpha$ is the same as the one of the orientation on $X$ then we choose an element $h\in G_{o(\alpha)}$ such that there exists $e\in E^+$ with $o(\alpha)=o(e)$ and $t(\alpha)=ht(e)$ (see Figure \ref{fig:same_orientation}). Hence this edge ends in $hT$.
  	\item If the direction of the edge $\alpha$ is opposite to the one given by the orientation on $X$ then we choose an element $h\in G_{o(\alpha)}$ such that there exists $e\in E^+$ with $t(\alpha)=t(e)$ and $o(\alpha)=hg_e^{-1}o(e)$ (see Figure \ref{fig:opposite_orientation}). Hence $o(\alpha)=hT$.
 	\end{itemize}
\noindent 
 Consider a $2$-cell $s$ in $F$ and denote by $v_1,\dots,v_n$ its vertices and by $\alpha_1,\dots,\alpha_n$ its edges such that the edge $\alpha_i$ starts at the vertex $v_i$. As supposed before $\alpha_1$ belongs to $V$. Hence to $\alpha_1$ we can associate an element $h_1\in G$ has explained before. Then $h_1^{-1}v_2=\tilde{v}_2\in V$ and the edge $\tilde{e}_2=h_1^{-1}\alpha_2$ starts in $V$. As a consequence, we can associate to it the element $h_2$ chosen above. The vertex $v_3\in h_1h_2T$. Keeping doing this process, all the edges $h_i^{-1}h_{i-1}^{-1}\dots h_1^{-1}e_{i+1}$ for $2\leq i\leq n-1$ belong to $V$ and we can associate to them the element $h_{i+1}$ chosen above. Note that by construction the $h_i's$ are elements of some $G_v$ with $v\in V$. Finally $h_1\dots h_nT=T$ and so $h_1\dots h_n\in G_{v_1}$. Choosing such an element $g_s\in G_{v_1}$ gives us a relation $r_s$: $h_1\dots h_ng_s^{-1}=1$ among elements of $\{G_v\}_{v\in V}$.
 \end{enumerate}  

 \medskip\noindent  In top of the relations within the $G_v$ and the relations given by the $2$-cells, we have the relations that identify elements of two isotropy subgroups corresponding to adjacent vertices through the stabiliser of the edge. More precisely, consider an edge $e\in E^+$. Denote by $\iota_{o(e)}$ and $\iota_{t(e)}$ the inclusion of stabilisers: $\iota_{o(e)}: G_e\hookrightarrow G_{o(e)}$ and $\iota_{t(e)}: G_e\hookrightarrow G_{t(e)}$, where $G_e$ denotes the stabiliser of the edge $e$. We have:
  \[\iota_{o(e)}(g)=\iota_{t(e)}(g) \text{ for any }e\in E^+ \text{ and for any }g\in G_e. \]

\medskip \noindent To resume let state Brown's theorem in our context.
\begin{thm}[{\cite[Theorem 1]{Brown_presentation}}]\label{thm_Brown}
Let $G$ be a group acting on an oriented simply connected CW complex $X$ such that the action is orientation-preserving. We assume moreover that $E^+$ can be chosen as the set of edges of a tree of representatives $T$. Then $G$ is generated by the isotropy subgroups $\{G_v\}_{v\in V}$ and the relations are generated by:
\begin{enumerate}
\item\label{item1:Brownrel} the relations inside the $G_v$,
\item\label{item2:Brownrel} $\iota_{o(e)}(g)=\iota_{t(e)}(g)$ for any $e\in E^+ $ and for any $g\in G_e$,
\item \label{item3:Brownrel}the relations $r_s=1$ for any $s\in F$.
\end{enumerate}
\end{thm}

\paragraph{Illustration}
Consider the action of the diedral group $D_4$ on the oriented CW simply connected planar complex $X_s$ of Figure \ref{fig:CW_complex}.
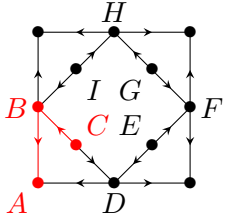
\begin{figure}
\begin{center}
	\begin{tikzpicture}
		\draw (0,2) node {$\bullet$};
		\draw (1,0) node{$\bullet$} node[below]
		{$D$};
		\draw (1,2) node{$\bullet$} node[above]{$H$};
		\draw (2,2) node{$\bullet$};
		\draw (2,1) node{$\bullet$} node[right]{$F$};
		\draw (2,0) node{$\bullet$};
		\draw (0.5,1.5) node{$\bullet$} node[below right]{$I$};
		\draw (1.5,0.5) node{$\bullet$} node[above left]{$E$};
		\draw (1.5,1.5) node{$\bullet$} node[below left]{$G$};
		\draw[postaction={mid arrow}, red] (0,1)--(0,0);
		\draw[postaction={mid arrow}] (0,1)--(0,2) ;
		\draw[postaction={mid arrow}] (2,1)--(2,2) ;
		\draw[postaction={mid arrow}] (1,2)--(0,2) ;
		\draw[postaction={mid arrow}] (1,0)--(0,0) ;
		\draw[postaction={mid arrow}] (1,2)--(2,2) ;
		\draw[postaction={mid arrow}] (1,0)--(2,0) ;
		\draw[postaction={mid arrow}] (2,1)--(2,0) ;
		\draw[postaction={mid arrow}] (0.5,1.5)--(0,1);
		\draw[postaction={mid arrow}, red] (0.5,0.5)--(0,1);
		\draw[postaction={mid arrow}] (1.5,1.5)--(2,1);
		\draw[postaction={mid arrow}] (1.5,0.5)--(2,1);
		\draw[postaction={mid arrow}] (0.5,1.5)--(1,2);
		\draw[postaction={mid arrow}] (1.5,1.5)--(1,2);
		\draw[postaction={mid arrow}] (0.5,0.5)--(1,0);
		\draw[postaction={mid arrow}] (1.5,0.5)--(1,0);
		\draw[red] (0,1) node {$\bullet$} node[left]{$B$};
		\draw[red] (0.5,0.5) node{$\bullet$} node[above right]{$C$};
		\draw[red] (0,0) node {$\bullet$} node[below left]{$A$};
	\end{tikzpicture}
\caption{The CW complex $X_s$.\label{fig:CW_complex}}
\end{center}
\end{figure}
The tree $T$ made of the vertices $A$, $B$ and $C$ with the two edges connecting them is a tree of representatives.
The isotropy subgroups $G_A$, $G_B$ and $G_C$ are isomorphic to $\Z_2$. Let call respectively $s_A$, $s_B$ and $s_C$ their generators.
The set $E^+$ is equal to the set of edges of $T$. There exists $2$ classes of $2$-cells whose representatives are given by the polygons $ABCD$ and $BCDEFGHI$. The relation given by $ABCD$ is $s_C=s_A$ and the one given by $BCDEFGHI$ is $s_Cs_Bs_Cs_Bs_Cs_Bs_C=s_B$. Finally, Brown's method gives us the following presentation of $D_4$:
\[\langle s_A,\  s_B,\ s_C\mid s_A^2,\  s_B^2,\  s_C^2,\  s_As_C^{-1}\ , (s_Cs_B)^4 \rangle\simeq \langle s_A,\  s_B\mid s_A^2,\  s_B^2,\ (s_As_B)^4 \rangle. \]

\subsection{Braided Higman-Thompson groups}\label{subsection:HT}
 \noindent In this section, we recall the construction of braided Higman-Thompson groups introduced in \cite{GLU_finiteness}, which generalised the constructions in \cite{Funar-Kapoudjian_brT_finitely_presented}. For integers $n,m\geq 2$, let $A_{n,m}$ be the infinite tree with one vertex of valence $m$ while all the other vertices have valence $n+1$,  embedded into the plane. We define the \emph{arboreal surface} $\mathscr{S}(A_{n,m})$ as the oriented planar surface with boundary obtained by thickening $A_{n,m}$ in the plane. Denote by $\mathscr{S}^\sharp(A_{n,m})$ the punctured arboreal surface obtained from $\mathscr{S}(A_{n,m})$ by adding a puncture for each vertex of the tree $A_{n,m}$. We fix a \emph{rigid structure} on $\mathscr{S}^\sharp(A_{n,m})$, that is, a decomposition of $\mathscr{S}(A_{n,m})$  into \emph{polygons} by a family of pairwise non-intersecting arcs whose endpoints lie on the boundary, in such a way  that each polygon contains exactly one vertex of the underlying tree in its interior and such that each arc crosses once and transversely a unique edge of the tree. The \emph{central polygon} is the unique polygon that has exactly $m$ arcs in its frontier if $m\neq n+1$. In the case  $m=n+1$, the central polygon is a polygon that we fix once for all.

	\medskip \noindent
	A subsurface $\Sigma$ of $\mathscr{S}^\sharp(A_{n,m})$ is called \emph{admissible} if it is a non-empty connected finite union of polygons that belong to the rigid structure. The \emph{frontier of $\Sigma$},  denoted by $\mathrm{Fr}(\Sigma)$, is defined as the union of the arcs defining the rigid structure that are contained in the boundary of $\Sigma$. A \emph{polygon adjacent} to $\Sigma$ is a polygon not contained in $\Sigma$ that shares an arc with the frontier of $\Sigma$.
	
	\noindent We call ahomeomorphism $\varphi : \mathscr{S}^\sharp(A_{n,m}) \to \mathscr{S}^\sharp(A_{n,m})$  \emph{asymptotically rigid} if the following conditions are satisfied:
	\begin{itemize}
		\item there is an admissible subsurface $\Sigma \subset \mathscr{S}^\sharp(A_{n,m})$ such that $\varphi(\Sigma)$ is admissible;
		\item the homeomorphism $\varphi$ is \emph{rigid outside $\Sigma$}, which means that the restriction \[\varphi : \mathscr{S}^\sharp(A_{n,m}) \backslash \Sigma \to \mathscr{S}^\sharp(A_{n,m}) \backslash \varphi( \Sigma)\]
		respects the rigid structure, i.e. it maps polygons to polygons. 
%		Such a surface $\Sigma$ is called a \emph{support} for $\varphi$.
	\end{itemize}
	We call the group of isotopy classes of orientation-preserving asymptotically rigid homeomorphisms of $\mathscr{S}^\sharp(A_{n,m})$  the \emph{braided Higman-Thompson group}. It is denoted by $\mathrm{br}T_{n,m}$. Let us emphasize that isotopies have to fix each puncture. The special instance $\mathrm{br}(T_{2,3})$ is exactly the group $T^\sharp$ introduced in \cite{Funar-Kapoudjian_brT_finitely_presented}. Figure \ref{fig:homeo} illustrates an element of the group $\mathrm{br}(T_{2,3})$.

\medskip\noindent In what follows, two particular kinds of elements of $\mathrm{br}T_{n,m}$ will be important as they will be generators of this group: twists and rotations.

\begin{ex}\label{ex_twist}
Let  $p_i$ and $p_j$ be  punctures of two adjacent polygons, and let $\Sigma$ to be the union of these two polygons. The element of $\Mod(\Sigma)$ twisting these punctures clockwise is called a \emph{twist}. We denote it by $\tau_{i,j}$ (see Figure \ref{fig:twist}).
\end{ex}

\begin{figure}
	\begin{center}
		\includegraphics[scale=0.4]{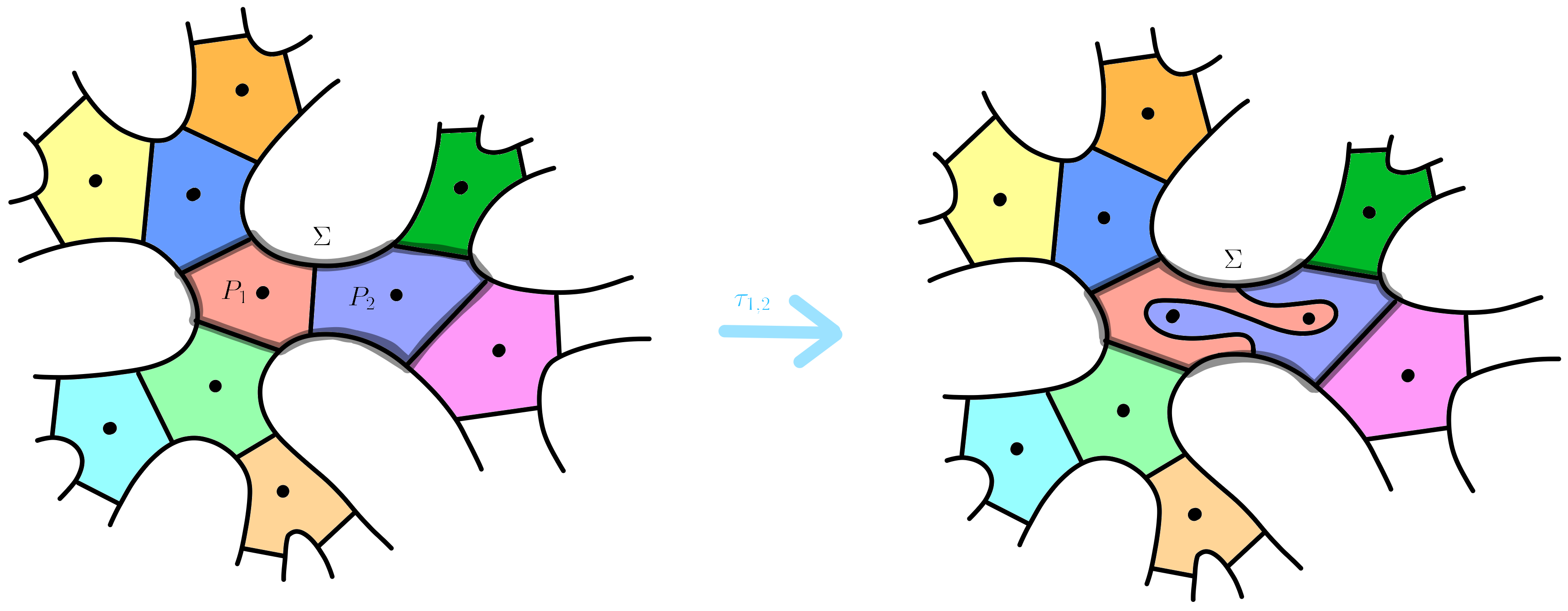}
		\caption{A twist in $\mathrm{mod}(\mathscr{S}^\#(A_{2,3}))$ \label{fig:twist}}
	\end{center}
\end{figure}
	
	\begin{ex}\label{ex_rotations}
	Let $\Sigma$ be any admissible subsurface containing the central polygon and exactly $k$ other polygons. The frontier of $\Sigma$ consists of exactly $m+k(n-1)$ arcs and so its complement in $\mathscr{S}^\sharp(A_{n,m})$ consists of $m+k(n-1)$ pairwise homeomorphic arboreal surfaces. Let $r_\Sigma$ be the asymptotically rigid homeomorphism that cyclically clockwise shifts the arcs of the frontier of $\Sigma$ (and hence the homeomorphic arboreal surfaces, without acting on them) and whose restriction to a disk in $\Sigma$ containing all the punctures is the identity (see Figure \ref{fig:rotation}). We call $r_\Sigma$ the \emph{rotation along $\Sigma$}. 
	\end{ex}

\begin{figure}
	\begin{center}
		\includegraphics[scale=0.4]{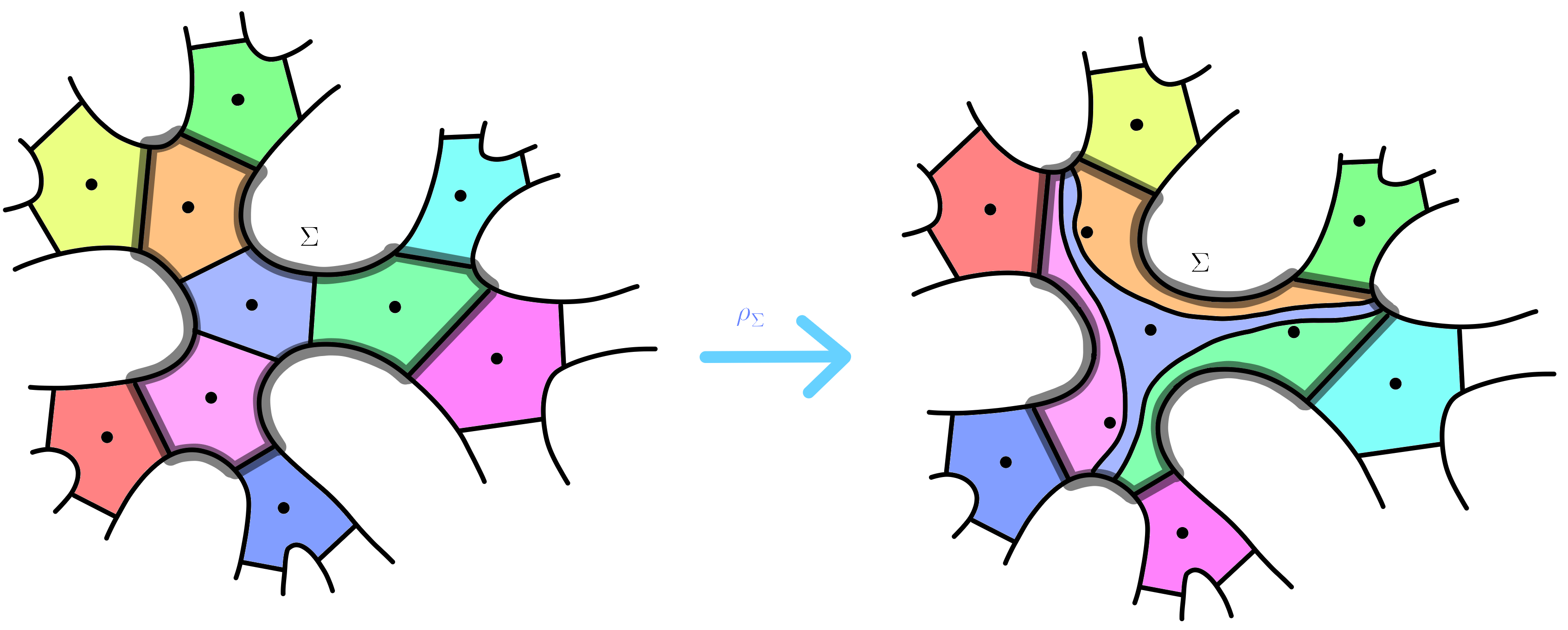}
		\caption{A rotation of  $\mathrm{mod}(\mathscr{S}^\#(A_{2,3}))$ \label{fig:rotation}}
	\end{center}
\end{figure}

\subsection{The spine cube complex}\label{subsection:spine} In \cite{GLU_finiteness}, the authors construct the \emph{spine complex}, a contractible cube complex on which $\mathrm{br}T_{n,m}$ acts. It is denoted by  $\mathscr{SC}(A_{n,m})$,  for $m,n\geq 2$. We recall this construction.

	\medskip\noindent A \emph{vertex} of $\mathscr{SC}(A_{n,m})$ is an equivalence class of a pair $(\Sigma,\varphi)$ consisting of an admissible subsurface containing the central polygon $\Sigma \subset \mathscr{S}^\sharp(A_{n,m})$  and an asymptotically rigid homeomorphism $\varphi : \mathscr{S}^\sharp(A_{n,m}) \to \mathscr{S}^\sharp(A_{n,m})$. The equivalence relation is given by: $(\Sigma_1, \varphi_1) \sim (\Sigma_2, \varphi_2)$ if $\varphi_2^{-1} \varphi_1$ is isotopic to an asymptotically rigid homeomorphism that maps $\Sigma_1$ to $\Sigma_2$ and that is moreover rigid outside $\Sigma_1$. We denote by $[\Sigma,\varphi]$ the vertex of $\mathscr{SC}(A_{n,m})$ that is represented by $(\Sigma, \varphi)$.
	
	\noindent If $[\Sigma, \varphi]$ is a vertex and if $H_1, \ldots, H_k$ are pairwise distinct polygons adjacent to $\Sigma$, we fill the subgraph spanned by \[\left\{ \left[  \Sigma \cup \bigcup\limits_{i \in I} H_i , \varphi \right] \mid I \subset \{1, \ldots, k\} \right\}\] with a \emph{$k$-cube}.

\medskip \noindent
The asymptotically rigid mapping class group $\mathrm{br}T_{n,m}$ acts on the spine cube complex $\mathscr{SC}(A_{n,m})$ by isometries: for an asymptotically rigid homeomorphism $g \in \mathrm{br}T_{n,m}$ and for a vertex $[\Sigma, \varphi] \in \mathscr{SC}(A_{n,m})$, we define
$$g \cdot [\Sigma, \varphi]:= [\Sigma, g \varphi].$$

 \noindent
Let us observe that, if $[\Sigma_1,\varphi_1]=[\Sigma_2,\varphi_2]$, then two surfaces $\Sigma_1$ and $\Sigma_2$ have to be homeomorphic, so they have the same number of punctures. With this, we define the \emph{height} of a vertex $x=[\Sigma, \varphi]$ as the \emph{height of $\Sigma$}, which is the number of punctures contained in $\Sigma$; we denote the height of $x$ by $h(x)$. Notice that, by construction of the complex $\mathscr{SC}(A_{n,m})$, if $x$ and $y$ are two adjacent vertices then we have $h(y)=h(x) \pm 1$. Hence, the edges of $\mathscr{SC}(A_{n,m})$ are naturally oriented by the height function from small to large height. Notice as well that the action of $\mathrm{br}T_{n,m}$ preserves the height function.

\noindent Later we will need the following lemma (note that in \cite{GLU_finiteness} it  was stated for the full cube complex instead of just the spine cube complex):

\begin{lemma}[{\cite[Lemma 4.2]{GLU_finiteness}}]\label{lemma_stab}
	The stabiliser in $\mathrm{br}T_{n,m}$ of a vertex $[\Sigma, \id]$ in $\mathscr{SC}(A_{n,m})$ is a subgroup of $\mathrm{stab}(\Sigma)$ in $\mathrm{Mod}(\mathscr{S}^\sharp(A_{n,m}))$, and it satisfies
	$$1 \to \mathrm{Mod}(\Sigma) \to \stab([\Sigma, \id]) \to \mathbb{Z}_{r(\Sigma)} \to 1$$
	for some integer $r(\Sigma) \geq 0$, where the morphism to $\mathbb{Z}_{r(\Sigma)}$ comes from the action by cyclic permutations of $\mathrm{stab}([\Sigma,\id])$ on components of $\mathrm{Fr}(\Sigma)$. 
\end{lemma}

\subsection{A presentation of the braid group}\label{Subsection_braids}
\noindent By Lemma \ref{lemma_stab}, stabilisers of vertices are semi-direct products of braid groups and cyclic groups. We will use the following presentation of braid groups stated only in the tree case.

\begin{thm}[{\cite{Sergiescu_presentation_tresses}}]\label{thm_presentation_braid_group} Let $\Gamma$ be a planar locally finite tree. The braid group associated to $\Gamma$ has the following presentation: it is generated by the edges of $\Gamma$ and the relations are generated by three types of relations:
	\begin{itemize}
		\item disjunction: if $\sigma_1$ and $\sigma_2$ are two disjoint edges, then $\sigma_1\sigma_2=\sigma_2\sigma_1$,
		\item adjacency: if the eges $\sigma_1$ and $\sigma_2$ have a common vertex, then: $\sigma_1\sigma_2\sigma_1=\sigma_2\sigma_1\sigma_2$,
		\item nodal: if the three edges $\sigma_1$, $\sigma_2$ and $\sigma_3$ have a unique common vertex and are clockwise ordered, then  $\sigma_1\sigma_2\sigma_3\sigma_1=\sigma_2\sigma_3\sigma_1\sigma_2=\sigma_3\sigma_1\sigma_2\sigma_3$.
	\end{itemize}
\end{thm}

\subsection{A simply connected subcomplex of bounded height}\label{subsection:spine_height}
Let $\mathcal{C}$ be a cube complex equipped with a height function. For each $k\geq 1$ we denote by $\mathcal{C}_{\leq k}$ the subcomplex of $\mathcal{C}$ generated by the vertices of height $\leq k$. 
We are interested in the complex $\mathscr{SC}_{\leq k}(A_{n,m})$, where $n,m \geq 2$ and  $\mathscr{SC}(A_{n,m})$ is the spine complex. Because the action of  $\mathrm{br}T_{n,m}$ on $\mathscr{SC}(A_{n,m})$ preserves the height, it induces an action of $\mathrm{br}T_{n,m}$ on $\mathscr{SC}_{\leq k}(A_{n,m})$. 

\medskip \noindent The goal of this section is to find small values of $k$ such that $\mathscr{SC}_{\leq k}(A_{n,m})$ is still simply connected in order to reduce the number of relations that we need to compute to obtain a presentation of $\mathrm{br}T_{n,m}$. 

\begin{prop}\label{prop_subcplx_sconnected}
The complex $\mathscr{SC}_{\leq 6}(A_{2,2})$ is simply connected and $\mathscr{SC}_{\leq 5}(A_{n,m})$ is simply connected for $(n,m)\neq (2,2)$.
\end{prop}

\noindent The following lemma is well known and allows us to reduce the proof of Proposition \ref{prop_subcplx_sconnected} to the study of the simple connectedness of descending links of vertices of height $k$ in $\mathscr{SC}_{\leq k}(A_{n,m})$.

\begin{lemma}\label{lemme_Morse_theory}
	Let $\mathcal{C}$ be a simply connected cube complex with a height function and let $k\in\Z$. If the descending link of every vertex of height $\geq k$ is simply connected then $\mathcal{C}_{\leq k-1}$ is simply connected
\end{lemma}

\begin{proof}
	Consider a loop $\gamma$ inside $\mathcal{C}_{\leq k-1}$. Up to homotopy, we can suppose that it has no backtracks and that it lies in the $1$-skeleton of $\mathcal{C}$. Because $\mathcal{C}$ is simply connected, there exists a combinatorial disk $D$ in $\mathcal{C}$ with boundary $\gamma$; we may assume that $D$ is contained in the $2$-skeleton of $\mathcal{C}$. Let $v\in D$ be a vertex of maximal height. If this height is $\leq k-1$ then $\gamma$ is already contractible in $\mathcal{C}_{\leq k-1}$ and we are done. So let us assume that the height of $v$ is $n\geq k$ and hence that its descending link is simply connected. 
Consider all its neighbourhood of smaller height that are in $D$. Consider a loop $\ell$ in the descending link of $v$ that passes only through theses vertices. By definition, this means that there exists a loop $\gamma'$ made of vertices of height $n-1$ and $n-2$ that is the boundary of a subdisk $D'$ of $D$. Because the descending link of $v$ is simply connected, there exists a combinatorial disk $L$ made of triangle in the descending link of $v$ with boundary $\ell$.  By definition, a triangle in the descending link of $v$ corresponds to a cube spanned by $v$ and three of the vertices in its descending link. Hence, $\gamma'$ is also the boundary of a combinatorial disk $D''$ made of vertices of height $<n$, and $D'$ and $D''$ are homotopic. Replacing $D'$ by $D''$ in $D$ do not change the boundary $\gamma$
but now either the maximal height of the vertices of the disk has decreased or the number of vertices of maximal height has decreased. We can continue this process until the disk is inside $\mathcal{C}_{k\leq{k-1}}$. \end{proof}

\noindent Let us recall the description from \cite{GLU_finiteness} of the descending links of $\mathscr{SC}(A_{n,m})$. Fix a disc $\mathbb{D}$ with $p\geq 1$ punctures in its interior and $q\geq 1$ marked points on its boundary. Let $P$ denote the set of punctures and $M=\{m_i \mid i \in \mathbb{Z}_q \}$ denote the set of marked points, ordered cyclically. From now on, an arc in $\mathbb{D}$ refers to an arc that starts from a marked point and that ends at a puncture. Given an arc $\alpha$, $\alpha(0)$ denotes the marked point it starts at, and $\alpha(1)$ denotes the puncture it ends at.

\medskip \noindent
Let $r\geq0$. Two arcs starting from the marked points $m_i$ and $m_j$ respectively, are \emph{$r$-separated} if they are disjoint and if the distance between $i$ and $j$ in $\mathbb{Z}_q$ is $>r$ (where $\mathbb{Z}_q$ is metrically thought of as the cycle $\mathrm{Cayl}(\mathbb{Z}_q,\{1\})$). Notice that being $0$-separated amounts to being disjoint. We define $\mathfrak{C}(p,q,r)$ as the simplicial complex whose vertices are the isotopy classes of arcs and whose simplices are collections of arcs that are pairwise $r$-separated (up to isotopy).

\noindent The following proposition is the main tool for the proof of Proposition~\ref{prop_subcplx_sconnected}. In \cite[Proposition~5.16]{GLU_finiteness}, we showed that for each $k$ and for $p, q, r$ large enough, the complex $\mathfrak{C}(p,q,r)$ becomes $k$-connected. The following proposition gives optimal bounds for $p,q,$ and $r$ such that $\mathfrak{C}(p,q,r)$ is simply connected. 

\begin{prop}\label{prop_simply_connected}
	The complex of arcs $\mathfrak{C}(p,q,r)$ is simply connected if $p\geq 5, q\geq 4r+3+\lceil \frac{r}{2}\rceil$, $r\geq 1$.
\end{prop}

\noindent The end of this subsection is dedicated to prove this proposition. We will be interested in complexes obtained by filling in certain punctures or removing marked points from the boundary. For this reason, we introduce the following complexes, which were already used in \cite{GLU_finiteness}. Let $\sim$ be a symmetric relation on $M$.  We denote by $\mathfrak{R}(\mathbb{D}\setminus P,P,M,\sim)$ the following simplicial complex: the vertices of $\mathfrak{R}$ are the isotopy classes of arcs in $\mathbb{D}\setminus P$ connecting a point in $M$ to a point in $P$, and its simplices are collections of arcs that are pairwise disjoint and that start from marked points that are pairwise $\sim$-related. Note that if $\sim$ is the relation of being $r$-separated, then $\mathfrak{R}(\mathbb{D}\setminus P,P,M,\sim)=\mathfrak{C}(p,q,r)$.

\begin{lemma}\label{lemma_arcs_less_intersection}
	Consider the complex of arcs $\mathfrak{R}(\mathbb{D}\setminus P,P,M,\sim)$. Let $\alpha$ and $\beta$ be two arcs having at least an intersection point outside the extremities. There exists an arc $\alpha'$ intersecting $\alpha$ only in its both extremities and such that the number of intersection between $\alpha'$ and $\beta$ is strictly less than the one between $\alpha$ and $\beta$.
\end{lemma}
\begin{proof}
We may assume that $\alpha$ and $\beta$ intersect in finitely many points. Now, let $a\in\alpha\cap\beta$ be such that the subarc of $\beta$ between $a$ and $\beta(1)$ does not intersect $\alpha$ anymore. Let $\alpha'$ be an arc from $\alpha(0)$ to $\alpha(1)$ following (but not intersecting) very closely $\alpha$ until it reaches $a$, then following (but not intersecting) $\beta$, until $\beta(1)$ if $\beta(1)=\alpha(1)$, or otherwise in the direction of $\beta(1)$, go around the puncture $\beta(1)$, following $\beta$ on the other side, and then following $\alpha$ all the way to $\alpha(1)$. By construction, the number of intersection between $\alpha'$ and $\beta$ is strictly less than the one between $\alpha$ and $\beta$.
\end{proof}

\begin{lemma}\label{lemma_relation}
	Consider the complex of arcs $\mathfrak{R}(\mathbb{D}\setminus P,P,M,\sim)$. Assume that $\lvert P\rvert \geq 3$. If the relation $\sim$ satisfies the following: for all $m,n\in M$, either there exist $m',n'\in M$ such that $m\sim m'$, $m'\sim n'$,  and  $n'\sim n$, or there exists a $m'\in M$ such that $m\sim m'$ and $m'\sim n$, 
	then $\mathfrak{R}(\mathbb{D}\setminus P,P,M,\sim)$ is connected.
\end{lemma}

\begin{proof}
	Let $\alpha$ be an arc from a marked point $m\in M$ to a puncture $p\in P$ and $\beta$ an arc from $n\in M$ to $q\in P$. 

	\medskip\noindent
	{\bf Case A:}\emph{ $\alpha$ and $\beta$ do not intersect, except possibly in their marked point if $m=n$.} If there exists $m'\in M$ $\sim$-related to both $m$ and $n$, consider an arc $\gamma$ from $m'$ to a puncture in $P\setminus\{p,q\}$ that does not intersect neither $\alpha$ or $\beta$. Then the class of $\gamma$ is connected in the complex to both the classes of $\alpha$ and $\beta$. If it is not the case, then there are $m',n'\in M$ such that $m\sim m'$, $m'\sim n'$,  and  $n'\sim n$. Let $\gamma_1$ be an arc from $m'$ to $q$ not intersecting $\alpha$, let $\gamma_2$ be an arc from $n'$ to $p$ not intersecting $\beta$. Then $\alpha$ and $\gamma_1$, as well as $\gamma_1$ and $\gamma_2$ and also $\gamma_2$ and $\beta$ are connected by edges. 
	
	\medskip\noindent
	{\bf Case B:} \emph{$m=n$}. We distinguish two subcases.
	\begin{itemize}
		\item Case B1: Assume $p=q$ and that $\alpha$ and $\beta$ do not intersect outside their extremities. Let $r\in P\setminus\{p\}$ and let $\gamma_1$ be an arc from $m$ to $r$ not intersecting neither $\alpha$ nor $\beta$ outside $m$. Then the class of $\gamma_1$ is connected to both the classes of $\alpha$ and $\beta$ by Case A. 
		
		\item Case B2: General case. We may assume that $\alpha$ and $\beta$ intersect in finitely many points. We will do by induction on $k$ the number of intersection between $\alpha$ and $\beta$ outside their extremities. The case $k=0$ is covered by Case B1 if $p=q$ or by Case A otherwise. Now, by Lemma~\ref{lemma_arcs_less_intersection}, there exists an arc $\alpha'$ whose class is connected to the class of $\alpha$ by Case B1 and to the class of $\beta$ by induction.
	\end{itemize}

	\medskip\noindent
{\bf Case C:} \emph{$m\sim n$}. We distinguish two subcases.
\begin{itemize}		
		\item Case C1: Assume that $p=q$ and that $p$ is the only intersection point of $\alpha$ and $\beta$. Let $r,r'\in P\setminus\{p\}$ be two distinct punctures. Let $\gamma_1$ and $\gamma_2$ be two arcs joining respectively
		$m$ to $r$, and $n$ to $r'$, and intersecting respectively $\alpha$ only in $m$ and $\beta$ only in $n$, such that $\gamma_1$ and $\gamma_2$ do not intersect.
		Then by Case A, the classes of $\alpha$ and $\gamma_1$, and of $\gamma_2$ and $\beta$ are connected. Moreover, the classes of $\gamma_1$ and $\gamma_2$ are connected by an edge.

		\item Case C2: General case. We may assume that $\alpha$ and $\beta$ intersect in finitely many points. We will do by induction on $k$ the number of intersection between $\alpha$ and $\beta$ outside their extremities. The case $k=0$ is covered by Case C1 if $p=q$ and by Case A otherwise. 
		Now, by Lemma~\ref{lemma_arcs_less_intersection}, there exists an arc $\alpha'$ whose class is connected to the class of $\alpha$ by Case B1 and to the class of $\beta$ by induction.
	\end{itemize}
\medskip\noindent
{\bf Case D:} \emph{$m\neq n$}. If there exists $m'\in M$ $\sim$-connected to both $m$ and $n$, consider two arcs $\gamma_1$ and $\gamma_2$ joining $m'$ to respectively $p$ and $q$. By Case C2, the classes of $\alpha$ and $\gamma_1$, and of $\beta$ and $\gamma_2$ are connected, and by Case B2, the one of $\gamma_1$ and $\gamma_2$ are also connected. Otherwise, there exists $m',n'\in M$ such that $m\sim m'$, $m'\sim n'$,  and  $n'\sim n$. In this case, consider two arcs $\gamma_1$ and $\gamma_2$ joining respectively $m'$ and $p$, and $n'$ and $q$. Then by Case C2, the classes of $\alpha$ and $\gamma_1$, of $\gamma_1$ and $\gamma_2$ and of $\beta$ and $\gamma_2$ are connected.
\end{proof}

\noindent Fix a set of punctures $P'\subset P$, a set of marked points $M' \subseteq M$. In what follows, we always consider the complex $\mathfrak{R}(\mathbb{D}\setminus P',P',M',\sim)$, where $\sim$ is the relation on $M'$ induced by the relation of being $r$-separated in $M$. Let us make the following  technical observation, which is a consequence of Lemma~\ref{lemma_relation}.

\begin{lemma}\label{lem:consecutive_connected}
If the cardinality of $M'$ is at least $2r+2$, the one of $P'$ is at least $3$ and two consecutive points of $M'$ are $\sim$-related then $\mathfrak{R}(\mathbb{D}\setminus P',P',M',\sim)$ is connected.
\end{lemma}

\begin{proof}
Let $k=\lvert M' \rvert$. And rename by $m_1,\dots, m_k$ the consecutive marked point in $M'$ such that $m_1$ and $m_k$ are $\sim$-related. Because $k\geq 2r+2$ then for any $m\in M'$, there exist $n\in M'$ such that $m\sim n$. 
Take $p,q\in M'$. Either there exists $m'\in M'$ such that $p\sim m'$ and $q\sim m'$, or because $k\geq 2r+2$ we have that either $p\sim m_1$ and $q\sim m_k$, or $p\sim m_k$ or $q\sim m_1$. In any case, the set of marked points $M'$ satisfies the conditions of Lemma \ref{lemma_relation}, hence $\mathfrak{R}(\mathbb{D}\setminus P',P',M',\sim)$ is connected.
\end{proof}

\begin{proof}[Proof of Proposition \ref{prop_simply_connected}]
	 Fix a puncture $p\in P$ and a marked point $m\in M$. We define $\mathcal{R}_{-1}$ to be the subcomplex of $\mathfrak{C}(p,q,r)$ generated by the vertices corresponding to the arcs connecting marked points that are $r$-separated from $m$ to punctures in $P\setminus \{p\}$. 
	 
	 \medskip\noindent The first step consists in proving that the inclusion of $R_{-1}$ in $\mathfrak{C}(p,q,r)$ induces an isomorphism on the fundamental groups. 
	 For $0\leq k\leq 2r$, we said that an arc is of \emph{type $k$} if it connects the marked point $m_k:=m+(-1)^{k+1}\lceil\frac{k}{2}\rceil$ to a puncture in $P\setminus \{p\}$, and of type \emph{$2r+1$} if it ends in the puncture $p$. For $0\leq k\leq 2r+1$, we define inductively the subcomplex $\mathcal{R}_{k}$ of $\mathfrak{C}(p,q,r)$ generated by the subcomplex $R_{k-1}$ and by the classes of arcs of type $k$. Note that $R_{2r+1}$ is the whole complex $ \mathfrak{C}(p,q,r)$. Note also that two vertices of the same type are never adjacent in the complex, hence $\mathcal{R}_{k}$ is obtained from $\mathcal{R}_{k-1}$ by gluing cones over the link in $\mathcal{R}_{k-1}$ of vertices of type $k$. Hence, it remains to show the following claim.

\begin{claim}\label{claim_link_RI_connected}
	 For $0\leq k\leq 2r+1$, the link inside $R_{k-1}$ of a vertex $\alpha_k$ of type $k$ is connected.
\end{claim}
\begin{proof}
 For $0\leq k\leq 2r+1$, the link of $\alpha_k$ in $R_{k-1}$ is isomorphic to $\mathfrak{R}(\mathbb{D}\setminus P'_k,P_k',M_k',\sim)$ where $P_k'=P\setminus \{p,\alpha_k(1)\}$ and for $0\leq k\leq 2r+1$, $M_k'$ consists of the marked points of $M$ that are $\sim$-related to $m$ and to $\alpha_k(0)$ together with the marked points $m_{k-1},m_{k-3},\dots m_{r-(k-r-1)}$ if $r+1\leq k\leq2r$, and $M_{2r+1}'$ consists of the marked points of $M$ that are $\sim$-related to $\alpha_{2r+1}$. Note that for all  $0\leq k\leq 2r+1$, $\lvert P_k'\rvert \geq 3$. 
We now distinguish three cases:
 \begin{itemize}
 	\item Because $\lvert M\rvert \geq  4r+3+\lceil \frac{r}{2}\rceil$, $\lvert M'_0\rvert=\lvert M'_{2r+1}\rvert \geq 2r +2 + \lceil \frac{r}{2}\rceil \geq 2r+2$ and in each case, the two points who are at distance (in $M$) exactly $r+1$ of respectively $m$ and $\alpha_{2r+1}(0)$ are $\sim$-related.
 	\item For $1\leq k\leq r$, $\lvert M'_k\rvert=\lvert M'_{0}\rvert-\lceil \frac{k}{2}\rceil  \geq 2r +2 + \lceil \frac{r}{2}\rceil -\lceil \frac{k}{2}\rceil\geq 2r+2$. Moreover the marked point that is at distance (in $M$) $r+1$ from $m$ and that is $\sim$-separated from $m_k$ and the marked point that is $\sim$-separated from $m$ and at distance (in $M$) exactly $r+1$ from $m_k$ are $\sim$-separated.
 	\item For $r+1\leq k\leq 2r$, $\lvert M'_k\rvert=\lvert M'_{0}\rvert-\lceil \frac{k}{2}\rceil + k-r\geq 2r +2 + \lceil \frac{r}{2}\rceil -\lceil \frac{k}{2}\rceil\ +k-r\geq 2r+2$.
 	Moreover the marked point that is at distance (in $M$) $r+1$ from $m_k$ and that is $\sim$-separated from $m$ and the marked point $m_{r-(k-r-1)}$ are $\sim$-separated.
 \end{itemize}
 
\noindent Hence, applying Lemma \ref{lem:consecutive_connected}, we obtain that the link inside $R_{k-1}$ of a vertex $\alpha_k$ of type $k$ is connected.
\end{proof}

\noindent As a consequence of Claim \ref{claim_link_RI_connected}, we can study the simply-connectedness of $\mathfrak{C}(p,q,r)$ by considering a loop lying in $R_{-1}$. Fix $\beta$ a simple arc connecting $m$ to $p$. Consider a loop $L$ in $R_{-1}$ we want to homotope it into the star of $\beta$. Since it is contractible, this will end the proof.

\noindent The arcs $\{\alpha_i\}_{1\leq i\leq n}$ representing the vertices of $L$ have their final points distinct from $p$ and their starting point $r$-separated from $m$, but they may intersect $\beta$. If there is no such intersection, then the vertices of $L$ already lie in the star of $\beta$, so there is nothing to prove in this case. Otherwise, let $1\leq j\leq n$ such that $\alpha_j$ is the arc that intersects $\beta$ the closest to $p$. Fix a small disc $D \subseteq S$ containing $p$ such that $D \cap \alpha_j$ is a subarc contained in $\partial D$ and such that $D$ is disjoint from all the $\alpha_i$ for $i\neq j$. Now let $\alpha'$ denote the arc obtained from $\alpha_j$ by replacing the subarc $\alpha_j \cap \partial D$ with $\partial D \backslash \alpha_j$. Notice that the vertex represented by $\alpha'$ is still connected to the vertices represented by $\alpha_{j-1}$ and by $\alpha_{j+1}$. Moreover the intersection of the links in $\mathfrak{C}(p,q,r)$ of $\alpha_j$ and $\alpha'$ is isomorphic to $\mathfrak{R}(\mathbb{D}\setminus P',P',M',\sim)$ where $P'=P\setminus\{\alpha_j(1)\}$ and $M'=M\setminus\{\alpha_j(0)\}$. By Lemma \ref{lem:consecutive_connected}, this intersection is connected and so $L$ is homotopic to the path $L'$ in $R_{-1}$ whose vertices are the same except that the vertex represented by $\alpha_j$ has been replaced by the one represented by $\alpha'$. Notice that doing this procedure, the total number of intersections between $\beta$ and the arcs representing the vertices of $L'$ is smaller than the total number of intersections between $\beta$ and the arcs representing the vertices of $L$. By iterating the argument, we find a loop homotopic to $L$ and whose vertices lie in the star of $\beta$, as desired. This concludes the proof.
\end{proof}

\begin{proof}[Proof of Proposition \ref{prop_subcplx_sconnected}]
	As a consequence of \cite[Proposition 5.8]{GLU_finiteness}, the descending link of a vertex of height $k$ is isomorphic to: 
\[\left\{ \begin{array}{cl} \mathfrak{C}(k,m+ (k-1)(n-1),n- 1) & \text{if $k \geq m+1$} \\ \mathfrak{C}_{\leq k-1}(k,m+(k-1)(n-1),n-1) & \text{if $k \leq m$} \end{array} \right.. \]
	By Proposition \ref{prop_simply_connected}, the descending link of a vertex of height $\geq 7$ when $m=n=2$, and of height $\geq 6$ otherwise, is connected. Moreover, by 
{\cite[Proposition 5.2]{GLU_finiteness}} the spine complex $\mathscr{SC}(A_{n,m})$ is contractible for all $m,n\geq2$.
	Hence, we conclude by Lemma \ref{lemme_Morse_theory}, that $\mathscr{SC}_{\leq 6}(A_{2,2})$ and $\mathscr{SC}_{\leq 5}(A_{n,m})$ for $(n,m)\neq (2,2)$ are simply connected.
\end{proof}

	\section{A presentation of $\mathrm{br}T_{n,m}$ for $m, n\geq 2$}\label{section_presentation}Let $m,n\geq 2$, we set $h(2,2)=6$ and $h(n,m)=5$ otherwise.
	 Consider the action of the braided Higman-Thompson groups $\mathrm{br}T_{n,m}$ on the subcomplex of the spine complex $\mathscr{SC}_{\leq h(n,m)}(A_{n,m})$ generated by vertices of height at most $h(n,m)$. To shorten the definition, in what follows it will be denoted by $\mathscr{SC}_{\leq}(A_{n,m})$ The cube complex $\mathscr{SC}_{\leq}(A_{n,m})$ is oriented (the orientation is given by the height of vertices), and, according to Proposition \ref{prop_subcplx_sconnected}, it is simply-connected. Moreover, the action preserves the orientation. We follow the construction of \cite{Brown_presentation} step by step keeping its notations that we have recalled in Subsection \ref{subsection:Brown} (see Theorem \ref{thm_Brown}).

			\subsection{Set-up}\label{section:set-up}
			Before making the choices needed for Brown's method, we set once and for all the notations used in this article. We also state some preliminaries facts. 
			
			\begin{figure}
				\begin{center}
					\includegraphics[scale=0.5]{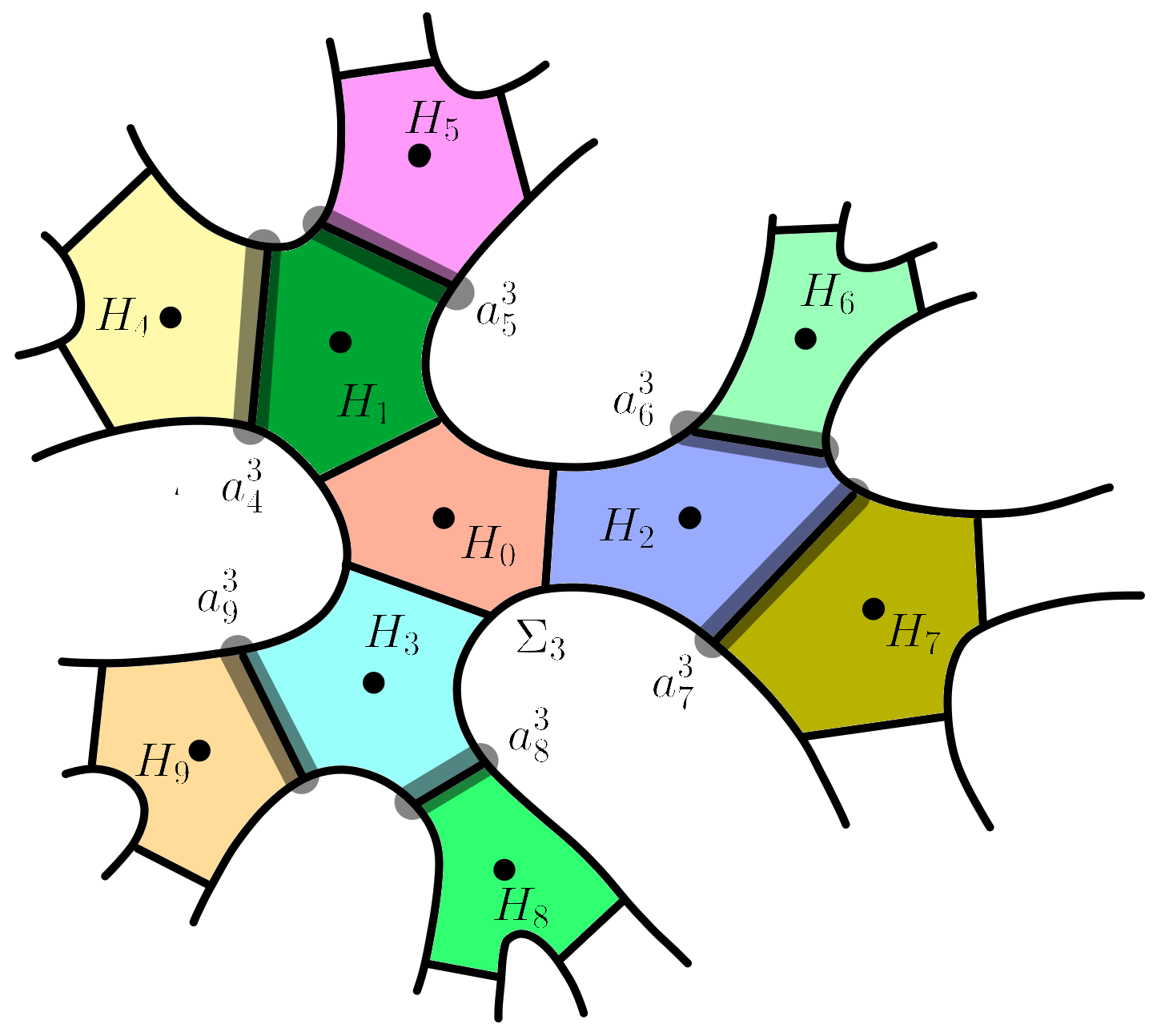}
					\caption{Arcs and polygons in $\mathrm{mod}(\mathscr{S}^\#(A_{2,3}))$ \label{fig:polygon}}
				\end{center}
			\end{figure}
			
			\medskip\noindent First, we choose inductively an ordered sequence of rigid polygons $\{H_k\}_{k\in\N}$ in $\mathscr{S}^\sharp(A_{n,m})$ and we denote by $p_i$ the puncture of $H_i$. Let $H_0$ be the central polygon and $H_1$ one of its adjacent polygon. If $H_k$ is defined for $k\geq 1$, $H_{k+1}$ is the next clockwise polygon adjacent to $H_0\cup H_1\cup \dots \cup H_{k-1}$ (see Figure \ref{fig:polygon}). For any $k\in \{0,\dots, h(n,m)-1\}$, we denote by $\Sigma_k$ the admissible subsurface of $\mathscr{S}^\sharp(A_{n,m})$ obtained as the union
			\[\Sigma_k:=\underset{0\leq i \leq k}{\bigcup}H_i.\] 
			Remark that the height of $\Sigma_k$ is $k+1$ and that it has $m+k(n-1)$ arcs in its frontier denoted by $\{a^k_i\}$ in such a way that $a^k_i=\Sigma_k\cap H_i$ for any polygon $H_i$ adjacent  to $\Sigma_k$.
			Because an element of $\mathrm{br}T_{n,m}$ preserving $\Sigma_k$ permutes cyclically its arcs, they will be indexed modulo $m+k(n-1)$ i.e. $a_i^k=a^k_{i+m+k(n-1)}$. We denote by
			\[\mathcal{I}_k=\{i\mid k+1\leq i\leq m+kn \}\] a complete set of representatives of the indices of the arcs of $\Sigma_k$.
			Note that $a_1^0=H_0\cap H_1$ while $a_1^1=a_{m+n}^1=H_1\cap H_{m+n}$.
			
			\medskip\noindent Consider a polygon $H$, not necessarily rigid, that is included either in $\Sigma_k$ or in its complementary. We denote by $\partial_k H$ the set of arcs of the polygon $H$ in the frontier of $\Sigma_k$: $\partial_k H=\Fr\Sigma_k\cap H$. Note that this set can be empty, can contain a single arc or several arcs.
			
				\begin{fact}\label{fact_arcs_Hk}
				For $k\geq 1$,  $\partial_kH_k=\{a_{m+(k-1)n+1}^k,\dots, a_{m+kn}^k\}$, and $\partial_0H_0=\{a_{1}^0,\dots, a_{m}^0\}$. Moreover, the indices of the arcs given belong respectively to $I_k$ and $I_0$.
			\end{fact}
			
				\begin{fact}\label{fact_arcs_intersection}
				If $i\in \mathcal{I}_{k}\cap \mathcal{I}_{\ell} $, then $a_{i}^{k}=a_{i}^{\ell}$.
			\end{fact}
			
		\noindent
			Recall that $r_{\Sigma_k}$ is the rotation around $\Sigma_k$ introduced in Example \ref{ex_rotations}. 
				\begin{fact}\label{fact_rotation_arc}
				For $i\in \mathcal{I}_k$, $r_{\Sigma_k}^j(a_{i}^k)=a_{i+j}^k$.
			\end{fact}
		\noindent We emphasize  that $i+j$ does not always belong to $\mathcal{I}_k$.
		Facts \ref{fact_arcs_Hk} and \ref{fact_rotation_arc} imply the following fact.
			\begin{fact}\label{fact_rotation}
		For $k\geq 1$, $\partial_kr_{\Sigma_k}^j(H_k)=\{a_{m+(k-1)n+1+j}^k,\dots, a_{m+kn+j}^k\}$.
		\end{fact}
	
	\noindent We denote by $e_k$ the edge linking $[\Sigma_k,\id]$ and $[\Sigma_{k+1},\id]$, and by $\stab([\Sigma_k,\id])$ the stabiliser of $[\Sigma_k,\id]$.
		\begin{fact}\label{fact_edge_facile}
			Let $e$ be an edge starting in $[\Sigma_k, \id]$ and ending in $[\Sigma_k\cup H_r,\id]$, for some $k+1\leq r\leq m+kn$. The rotation $r_{\Sigma_k}^{r-(k+1)}$ belongs to
			$\stab([\Sigma_k,\id])$ and it sends $H_{k+1}$ to $H_r$. Consequently, it sends the edge $e_k$ to the edge $e$.
		\end{fact}
			
\noindent
			Let $H_{r_1}$ and $H_{r_2}$ be two polygons adjacent to $\Sigma_k$. We define \emph{the distance between $H_{r_1}$ and $H_{r_2}$} as $r_1-r_2$ modulo $m+k(n-1)$. Finally, we recall that $\tau_{i,j}$ denotes the twist between the punctures $p_i$ and $p_j$ of two adjacent polygons (see Example \ref{ex_twist}). Notice that in the case $m=2$ (respectively $m=3$), the polygons $H_3$ and $H_4$ (respectively $H_4$) are not adjacent to $H_0$ but to $H_1$. In the same way if $(n,m)=(2,2)$ then $H_5$ is adjacent to $H_2$. Hence, to shorten the notations, we set for $1\leq i\leq 4$,  \[\tau_i:=\begin{cases}
		\tau_{0,i} & \text{if } i\leq m\\
		\tau_{1,i}& \text{ if } i>m
	\end{cases} \ \ \ \ \text{ and } \ \ \ \ \tau_5:=\tau_{2,5}.\]

					\subsection{Choices of representatives}
				Recall that we denote by $o(\alpha)$ the vertex of origin of an edge $\alpha$ of $\mathscr{SC}_{\leq}(A_{n,m})$ and by $t(\alpha)$ the terminal vertex. We need now to do several choices: 
				\begin{itemize}
				\item a choice of a \emph{tree of representatives} $T$, meaning a tree $T$ such that its set of vertices $V$ is a set of representatives of the vertices of $\mathscr{SC}_{\leq}(A_{n,m})$ under the action of $\mathrm{br}T_{n,m}$;
				\item a choice of a set of representatives (under the action of $\mathrm{br}T_{n,m}$) of edges $E^+$ starting in a vertex of $T$ and that contains the edges of $T$;
			\item a choice of representatives of squares $F$ based on vertices of $T$. 
				\end{itemize}

		\subsubsection{Choice of a tree of representatives}
		We choose the following tree $T$ as tree of representatives. The set of vertices $V$ of $T$ is $\{[\Sigma_k, \id]\}_{0\leq k\leq h(n,m)-1}$ and the set of edges of $T$ is $\{e_k\}_{0\leq k\leq h(n,m)-2}$  (see Figure \ref{fig:tree_rep}).
		 
		 \begin{figure}
		 	\begin{center}
		 		\begin{tikzpicture}
		 			\draw (0,0) node{$\bullet$} node[below]	{$[\Sigma_0, \id]$};
		 			\draw (3,0) node{$\bullet$} node[below]	{$[\Sigma_1, \id]$};
		 			\draw (6,0) node{$\bullet$} node[below]	{$[\Sigma_2, \id]$};
		 			\draw (9,0) node{$\bullet$} node[below]	{$[\Sigma_3, \id]$};
		 				\draw (12,0) node{$\bullet$} node[below]	{$[\Sigma_4, \id]$};
		 			\draw (0,0) -- (3,0) -- (6,0)-- (9,0)--(12,0);
		 			\draw (1.5,0) node[above]	{$e_0$};
		 			\draw (4.5,0) node[above]	{$e_1$};
		 			\draw (7.5,0) node[above]	{$e_2$};
		 			\draw (10.5,0) node[above]	{$e_3$};
		 		\end{tikzpicture}
		 	\end{center}
	 	 		\caption{Tree $T$: our choice of a tree of representatives of the action of $\mathrm{br}T_{n,m}$ on $\mathscr{SC}_{\leq}(A_{n,m})$.\label{fig:tree_rep}}
		 \end{figure}
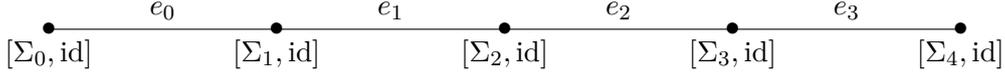
		
		\begin{lemma}\label{lemma_tree_rep}
		Let $n,m\geq 2$. $T$ is a tree of representatives for the action of $\mathrm{br}T_{n,m}$ on $\mathscr{SC}_{\leq}(A_{n,m})$.
		\end{lemma} 
		\begin{proof}
		 Let $[\Sigma, f]$ be a vertex of $\mathscr{SC}_{\leq}(A_{n,m})$. It is in the orbit of $[\Sigma, \id]$. Consider $\mathscr{S}^\sharp(A_{n,m})$ and cut it along the $m+(n-1)(h(\Sigma)-1)$ extremal arcs of $\Sigma$. We obtain  the surface $\Sigma$ and $m+(n-1)(h(\Sigma)-1)$ infinite surfaces $S_1, \dots S_{m+(n-1)(h(\Sigma)-1)}$ homeomorphic to $\mathscr{S}^\sharp(A_{n,n})$. Cutting $\mathscr{S}^\sharp(A_{n,m})$ along the extremal arcs of $\Sigma_{h(\Sigma)-1}$ gives the surface $\Sigma_{h(\Sigma)-1}$ and $m+(n-1)(h(\Sigma)-1)$ infinite surfaces $S'_1,\dots, S'_{h(\Sigma)+2}$ homeomorphic to $\mathscr{S}^\sharp(A_{n,n})$. The surfaces $\Sigma$ and $\Sigma_{h(\Sigma)-1}$ are homeomorphic as they have the same number of puncture $h(\Sigma)$. Consequently, there exists an homeomorphism $g$ of $\mathscr{S}^\sharp(A_{n,m})$ preserving the orientation, sending $\Sigma$ to $\Sigma_{h(\Sigma)-1}$ and $\{S_i\}_{1\leq i \leq m+(n-1)(h(\Sigma)-1)}$ to $\{S'_i\}_{1\leq i \leq m+(n-1)(h(\Sigma)-1)}$. Hence, $g\in \mathrm{br}T_{n,m}$ is rigid outside $\Sigma$, so the image of $[\Sigma, \id]$ by  the class of $g$ is $[\Sigma_{h(\Sigma)-1},\id]$ and so $[\Sigma_{h(\Sigma)-1},\id]$ belongs to the orbit of $[\Sigma, f]$. 
		 
	\medskip\noindent To conclude that $T$ is a tree of representatives we notice that the action of $\mathrm{br}T_{n,m}$ preserves the height so two distinct vertices of $T$ are not in the same orbit. 
		\end{proof}

		\subsubsection{Choice of a special set of representatives of edges}
	We choose a set $E^+$ of representatives of edges of $\mathscr{SC}_{\leq}(A_{n,m})$ under the action of $\mathrm{br}T_{n,m}$ containing all the edges of $T$ and starting at a vertex of the set $V$.
	
		\begin{lemma}\label{lemma_edge_T}
				Let $m,n\geq 2$,  $E^+=\{e_k\}_{0\leq k\leq h(n,m)-2}$.
		\end{lemma}
		\begin{proof}
			Consider an edge $e$ of $\mathscr{SC}_{\leq}(A_{n,m})$. By Lemma \ref{lemma_tree_rep}, we can assume that $o(e)\in V$, hence $o(e)=[\Sigma_k,\id]$. By \cite[Lemma 3.4]{GLU_finiteness}, $t(e)=[\Sigma_k\cup H,\id]$ where $H$ is a polygon adjacent to $\Sigma_k$. There exists a power of $r_{\Sigma_{k}}$ that sends $H$ to $H_{k+1}$. Hence, $e$ belongs to the orbit of $e_k$. On the other hand, the action preserving the height of the vertices, two edges $e_k$ are not in the same orbit.
		\end{proof} 

\subsubsection{Choice of a special set of representatives of squares}

We choose a set of representatives of $2$-cells of $\mathscr{SC}_{\leq}(A_{n,m})$ under the action of $\mathrm{br}T_{n,m}$ such that the representatives are based on a vertex of $V$. 

\medskip \noindent
Let $F$ be the set of squares spanned by the vertex $[\Sigma_k,\id]$ and the polygons $H_{k+1}$ and $H_r$, for $0\leq k\leq h(n,m)-3$ and for $k+2\leq r\leq k+1+\left\lceil\frac{m+(n-1)k-1}{2}\right\rceil$ (see Figure \ref{fig:square}).
		
\begin{figure}
	\begin{center}
		\begin{tikzcd}
			{[\Sigma_k\cup H_r,\id] } \arrow[d,"\alpha_4"] & 	{[\Sigma_{k+1}\cup H_r,\id] } \arrow[l,"\alpha_3"] \\ 
			{[\Sigma_k,\id] }\arrow[r, "\alpha_1=e_k"]	& {[\Sigma_{k+1},\id] } \arrow[u,"\alpha_2"] \end{tikzcd}
	\end{center}
	
	\caption{Squares of $F$: $0\leq k\leq h(n,m)-3$ and $k+2\leq r\leq k+ 1+\left\lceil\frac{m+(n-1)k-1}{2}\right\rceil$. \label{fig:square}}
\end{figure}

\begin{lemma}\label{lemma_rep_squares}
	Let $m,n\geq 2$. $F$ is a set of representatives of the squares of $\mathscr{SC}_{\leq}(A_{n,m})$ under the action of $\mathrm{br}T_{n,m}$.
\end{lemma}
\begin{proof}
Consider a square $C$ in $\mathscr{SC}_{\leq}(A_{n,m})$. We denote by $k$ the smallest height of its vertices. Note that $0\leq k\leq 2$. By Lemma \ref{lemma_edge_T}, we can assume that this square is generated by the vertex $[\Sigma_k,\id]$, and by two of its adjacent polygons $H_{k+1}$ and $H_s$, for some $k+2\leq s\leq m+nk$. If $s\leq k+1+\left\lceil\frac{m+(n-1)k-1}{2}\right\rceil$ then $C$ belongs to $F$. Otherwise, by Fact \ref{fact_rotation_arc}, $r_{\Sigma_k}^{-(s-(k+1))}$ sends $a^k_{s}$ to $a^k_{k+1}$, and $a^k_{k+1}$ to $a^k_{2(k+1)-s}=a^k_{\ell}$ for $\ell=2(k+1)-s+ m+k(n-1)$. Note that \[k+2\leq \ell\leq k+1+\left\lceil\frac{m+(n-1)k-1}{2}\right\rceil,\] and in particular $\ell\in I_k$. Hence, $r_{\Sigma_k}^{-(s-(k+1))}$ which belongs to $\stab([\Sigma_k,\id])$, sends $H_s$ to $H_{k+1}$ and $H_{k+1}$ to $H_{\ell}$. As a consequence, it sends the square $C$ to a square of $F$.
	
	\medskip\noindent On the other hand, two squares of $F$ whose the respective smallest heights of its vertices are different can not be in the same equivalence class. So consider two different squares $C_1$ and $C_2$ of $F$ both based on $[\Sigma_k,\id]$, for $0\leq k\leq 2$, and generated by $H_{k+1}$ and respectively by $H_{r_1}$ and $H_{r_2}$, for two distinct indices $r_1,r_2 \in  \{k+2, \dots, k+1+\left\lceil\frac{m+(n-1)k-1}{2}\right\rceil \}$. Assume they are in the same orbit and let $g\in \mathrm{br}T_{n,m}$ sending the square $C_1$ to $C_2$. Then $g$ has to fix the vertex $[\Sigma_k,\id]$ so there exists a representative of $g$ that preserves $\Sigma_k$. In particular it permutes cyclically the polygons adjacent to $\Sigma_k$, hence, it has to preserve the distance between $H_{k+1}$ and $H_{r_1}$, which is different from the distance between $H_{k+1}$ and $H_{r_2}$. Consequently the two squares $C_1$ and $C_2$ are not in the same orbit and this achieves the proof that $F$ is a set of representatives of squares.
\end{proof}

\subsection{Presentations of the vertex and edge stabilisers of the tree of representatives}
By Theorem \ref{thm_Brown}, to compute a presentation of  $\mathrm{br}T_{n,m}$, we need to obtain a presentation of the vertex stabilisers and to identify elements between the stabilisers of two adjacent vertices through the edge stabilisers.

\subsubsection{Isotropy subgroups of vertices} In this paragraph we study the presentations of the vertex stabilisers $\stab[\Sigma_k,\id]$, for $0\leq k\leq h(n,m)-1$.

\begin{prop}\label{lemma_presentation_Stab_Sigma_k}
	For $0\leq k\leq h(n,m)-1$, the subgroup $\stab[\Sigma_k,\id]$ is generated by $r_{\Sigma_k}$ and $\tau_i$ for $1\leq i\leq k$. The relations are generated by :
	\begin{itemize}
		\item the {braids relations:} 
			\begin{enumerate}
				\item when $m<k$, $\tau_i\tau_\ell=\tau_\ell\tau_i$ for any $2\leq i\leq m<\ell\leq \min(k,4)$,
				\item $\tau_i\tau_j\tau_i=\tau_j\tau_i\tau_j$ for any $1\leq i<j\leq \min(k,m)$,
				\item when $m<k	$,   $\tau_1\tau_\ell\tau_1=\tau_\ell\tau_1\tau_\ell$ for any $m<\ell \leq\min( k,m+n)$,
				\item $\tau_i\tau_j\tau_s\tau_i=\tau_j\tau_s\tau_i\tau_j=\tau_s\tau_i\tau_j\tau_s$ for any $1\leq i<j<s\leq \min(k,m)$,
				\item  when $m=2$ and $k\geq4$, $\tau_3\tau_4\tau_3=\tau_4\tau_3\tau_4$ and  $\tau_1\tau_{3}\tau_4\tau_1=\tau_{3}\tau_4\tau_1\tau_{3}=\tau_4\tau_1\tau_{3}\tau_4$,
				\item when $(n,m,k)=(2,2,5)$, $\tau_5\tau_i=\tau_i\tau_5$, for any $i\in\{1,3,4\}$ and $\tau_2\tau_5\tau_2=\tau_5\tau_2\tau_5$.
	\end{enumerate}

		\item the \emph{commutation relations}: $r_{\Sigma_k} \tau_{i}=\tau_{i}r_{\Sigma_k}$ for $1\leq i\leq k$.
	
		\item the \emph{rotation relation}: $r_{\Sigma_k}^{m+k(n-1)}=(\tau_{k}\tau_{k-1}\dots\tau_{1})^{-(k+1)}$.
	\end{itemize}
\end{prop}
\begin{proof}
	Using \cite[Lemma 4.2]{GLU_finiteness}, we have the following short exact sequence:
	\[1 \to \mathrm{Mod}(\Sigma_k) \to \stab([\Sigma_k, \id]) \to \mathbb{Z}_{m+k(n-1)} \to 1.\]
	We use the presentation of $\mathrm{Mod}(\Sigma_k)$ given by Theorem \ref{thm_presentation_braid_group}, hence $\stab[\Sigma_k,\id]$ is generated by $r_{\Sigma_k}$ and the $\tau_{i}$ for $1\leq i\leq k$. 
	
	\medskip\noindent The elements $r_{\Sigma_k}$ and $\tau_{i}$ commute because the first one fixes the punctures and permutes the element in $\Fr\Sigma_k$ whereas $\tau_{i}$ twist the punctures $0$ (respectively $1$ if $k>m$ and $2$ if $(n,m)=(2,2)$) and $i$, and fixes the elements in $\Fr\Sigma_k$. Consequently we have the commutation relations: \[r_{\Sigma_k} \tau_{i}=\tau_{i}r_{\Sigma_k}, \ \ \text{ for  }1\leq i\leq k.\]
	
	\medskip\noindent Using the commutation relations we obtain that a relation is of the form $r_{\Sigma_k}^iw=\id$ for some power $i$ and some $w\in \mathrm{Mod}(\Sigma_k)$. Note that because $\id$ fixes the elements in $\Fr\Sigma_k$, $i$ has to be a multiple of $m+k(n-1)$. Hence to obtain generators of the relations, we only need to consider the case $i=0$ and $i=1$. When $i=0$, relations are generated by the braids relations obtained from Theorem˜\ref{thm_presentation_braid_group} (see Figure \ref{fig:tree_sigmak} ). When $i=1$, relations are generated by the rotation relation. 
	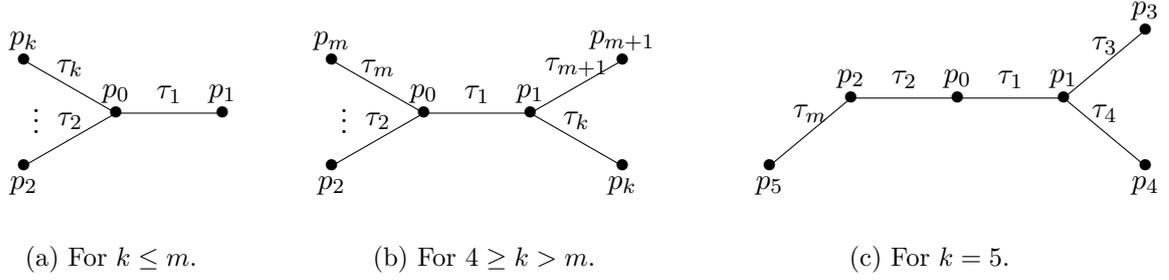
\begin{figure}
		\begin{subfigure}{.2\textwidth}
			\begin{center}
				\begin{tikzpicture}
					\begin{scope}[scale=0.7]
					\draw (0,0) node {$\bullet$} node[above] {$p_0$};
					\draw (2,0) node {$\bullet$} node[above] {$p_1$};
					\draw (-1.5,0) node {$\vdots$};
					\draw ({-2*cos(30)},{2*sin(30)}) node {$\bullet$} node[above] {$p_k$};
					\draw ({-2*cos(30)},{-2*sin(30)}) node {$\bullet$} node[below] {$p_2$};
					\draw (0,0) -- (2,0) node[midway, above]{$\tau_1$};
					\draw (0,0) -- ({-2*cos(30)},{2*sin(30)}) node[midway, above]{$\tau_k$};
					\draw (0,0) -- ({-2*cos(30)},{-2*sin(30)}) node[midway, above]{$\tau_2$};
					\end{scope}
				\end{tikzpicture}
			\end{center}
			\caption{For $k\leq m$.\label{fig:braid_relations_k<m}}
		\end{subfigure}
		\hfill
		\begin{subfigure}{.33\textwidth}
			\begin{center}
				\begin{tikzpicture}
					\begin{scope}[scale=0.7]
					\draw (0,0) node {$\bullet$} node[above] {$p_0$};
					\draw (2,0) node {$\bullet$} node[above] {$p_1$};
					\draw (-1.5,0) node {$\vdots$};
					\draw ({-2*cos(30)},{-2*sin(30)}) node {$\bullet$} node[below] {$p_2$};
					\draw ({-2*cos(30)},{2*sin(30)}) node {$\bullet$} node[above] {$p_m$};
					\draw (0,0) -- (2,0) node[midway, above]{$\tau_1$};
					\draw (0,0) -- ({-2*cos(30)},{-2*sin(30)}) node[midway, above]{$\tau_2$};
					\draw (0,0) -- ({-2*cos(30)},{2*sin(30)}) node[midway, above]{$\tau_m$};
					\draw ({2+2*cos(30)},{2*sin(30)}) node {$\bullet$} node[above] {$p_{m+1}$};
					\draw ({2+2*cos(30)},{-2*sin(30)}) node {$\bullet$} node[below] {$p_k$};
					\draw (2,0) -- ({2+2*cos(30)},{2*sin(30)}) node[midway, above]{$\tau_{m+1}$};
					\draw (2,0) -- ({2+2*cos(30)},{-2*sin(30)}) node[midway, above]{$\tau_k$};
						\end{scope}
				\end{tikzpicture}
			\end{center}
			\caption{For $4\geq k> m$.\label{fig:braid_relations_k>m}}
		\end{subfigure}
		\hfill
				\begin{subfigure}{.33\textwidth}
			\begin{center}
				\begin{tikzpicture}
					\begin{scope}[scale=0.7]
					\draw (0,0) node {$\bullet$} node[above] {$p_0$};
					\draw (2,0) node {$\bullet$} node[above] {$p_1$};
					\draw (-2,0) node {$\bullet$} node[above] {$p_2$};
					\draw ({-2-2*cos(40)},{-2*sin(40)}) node {$\bullet$} node[below] {$p_5$};
					\draw (0,0) -- (2,0) node[midway, above]{$\tau_1$};
					\draw (0,0) -- (-2,0) node[midway, above]{$\tau_2$};
					\draw (-2,0) -- ({-2-2*cos(40)},{-2*sin(40)}) node[midway, above]{$\tau_m$};
					\draw ({2+2*cos(40)},{2*sin(40)}) node {$\bullet$} node[above] {$p_{3}$};
					\draw ({2+2*cos(40)},{-2*sin(40)}) node {$\bullet$} node[below] {$p_4$};
					\draw (2,0) -- ({2+2*cos(40)},{2*sin(40)}) node[midway, above]{$\tau_{3}$};
					\draw (2,0) -- ({2+2*cos(40)},{-2*sin(40)}) node[midway, above]{$\tau_4$};
				\end{scope}
				\end{tikzpicture}
			\end{center}
			\caption{For $k=5$.\label{fig:braid_relations_k=5}}
		\end{subfigure}
		\caption{The subtree of $A_{n,m}$ inside $\Sigma_k$.\label{fig:tree_sigmak}}
	\end{figure}
	The rotation relation is a consequence of the fact that $r_{\Sigma_k}^{m+k(n-1)}$ fixes the punctures pointwise and has make done to any polygon inside $\Sigma_k$ a full twist. Hence, to undo this, $w$ has to be the inverse of a full twist of the punctures inside $\Sigma_k$. Note that 
	the braid $\tau_{k}\tau_{k-1}\dots\tau_{1}$ cyclically permutes clockwise the punctures $p_0,p_1,\dots,p_k$ if $k\leq m$, $p_1,p_2,\dots, p_m, p_0,p_{m+1},\dots, p_k$ if $4\geq k>m$ (see Figure\ref{fig:permutation_des pointes} and $p_0,p_3,p_4,p_1,p_5,p_2$ if $k=5$ (see Figure\ref{fig:permutation_des pointes2}.).

		\begin{figure}
		\begin{center}
			\begin{tikzcd}[cells={nodes={draw=black}}, row sep=small, column sep = small]
&&&&&p_{m+1}\arrow[drr, dash]\arrow[r, bend left=50,blue] & p_k \arrow[dr, dash]\arrow[dr, bend left=50,blue] \\
p_1 \arrow[dr, dash] \arrow[r, bend left=50,blue] & \dots\arrow[r, bend left=50,blue]\arrow[d, dash] & p_k\arrow[dl, dash]\arrow[dl, bend left=50,blue]&&&&	&	p_1 \arrow[dr, dash] \arrow[r, bend left=50,blue] &\dots\arrow[r, bend left=50,blue]\arrow[d, dash] & p_m\arrow[dl, dash]\arrow[dl, bend left=50,blue]\\
&p_0\arrow[ul, bend left=50,blue]&&&&&&&	p_0\arrow[uulll, bend left=50,blue]\\
			\end{tikzcd}
			\caption{$\tau_{k}\tau_{k-1}\dots\tau_{1}$ induces a clockwise cyclic permutations of the punctures\label{fig:permutation_des pointes}. On the left for $k\leq m$ and on the right for $4\geq k>m$.}
		\end{center}
	\end{figure}
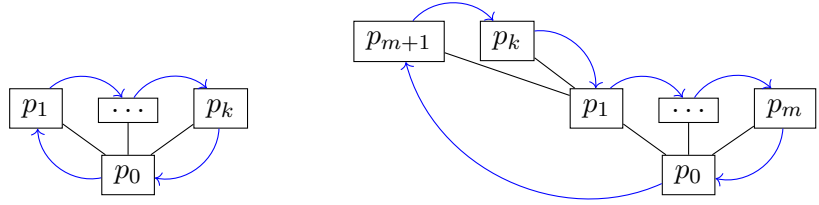

	\begin{figure}
	\begin{center}
		\begin{tikzcd}[cells={nodes={draw=black}},  row sep=small, column sep = small]
		p_{3}\arrow[drr, dash]\arrow[r, bend left=50,blue] & p_4 \arrow[dr, dash]\arrow[dr, bend left=50,blue] &&& p_{5}\arrow[d, dash]\arrow[d, bend left=50,blue] \\
		&	&	p_1 \arrow[dr, dash] \arrow[urr, bend left=50,blue] && p_2\arrow[dl, dash]\arrow[dl, bend left=50,blue]&&\\
		&&&	p_0\arrow[uulll, bend left=50,blue]&&&\\
		\end{tikzcd}
		\caption{$\tau_{5}\tau_{4}\dots\tau_{1}$ induces a clockwise cyclic permutations of the punctures\label{fig:permutation_des pointes2} in the case $(m,n)=(2,2)$.}
	\end{center}
\end{figure}
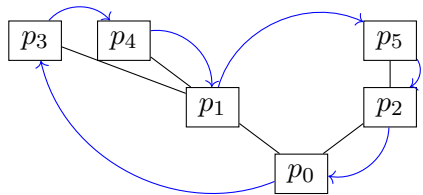

	\medskip\noindent As a consequence, the relations are generated by the ones of the braid groups, by the commutation relations and the rotation relation  announced.
\end{proof}

\subsubsection{Isotropy subgroups of edges}
\noindent In this paragraph we compute the edge stabilisers.
\begin{prop}\label{lemma_stab_edge}
	For $0\leq k\leq h(n,m)-2$, $\stab e_k$ is isomorphic to $\Mod(\Sigma_k)$.
\end{prop}

\begin{proof}
The action preserving the height of vertices, it does not inverse any edge, hence $\stab e_k=\stab[\Sigma_k,\id]\cap \stab[\Sigma_{k+1},\id]$.
	Note that $\Mod(\Sigma_k)$ is included in $\stab e_k$.
	
	\medskip\noindent An element $g\in \stab e_k$, is an element of $\stab[\Sigma_k,\id]$ that sends $H_{k+1}$ to itself. Using the presentation of $\stab[\Sigma_k,\id]$ from Lemma \ref{lemma_presentation_Stab_Sigma_k}, $g$ can be written as follows: $g=r_{\Sigma_k}^\l w$ where $w\in \Mod(\Sigma_k)$. Using the fact that $w\in \stab[\Sigma_{k+1},\id]$, we obtain that $r_{\Sigma_k}^\l\in \stab[\Sigma_{k+1},\id]$, and so $\l$ is a multiple of $m+k(n-1)$, hence using the rotation relation, $g$ can be written as a product of twists and their inverses. Consequently, $g\in \Mod(\Sigma_k)$ and $\stab e_k$ is isomorphic to $\Mod(\Sigma_k)$ as expected.
\end{proof}

\begin{remark}
	This lemma justifies that we took the same notation for a twist seen in $\Sigma_k$ and in $\Sigma_{k+1}$.
\end{remark}

\subsection{Construction of relations corresponding to squares}\label{Subsection:relation_squares}
\noindent The last step is to compute the relations given by the squares of $F$ (see Lemma \ref{lemma_rep_squares} and Figure \ref{fig:square}).

\medskip\noindent Following the proof of Brown (see Section \ref{subsection:Brown}) we will associate an element $h_i\in \stab([\Sigma_{h(o(\alpha_i))},\id])$ to each edge $\alpha_i$ depending on their orientation (see Figure \ref{fig:same_orientation} for $\alpha_1$ and $\alpha_2$, and Figure \ref{fig:opposite_orientation} for $\alpha_3$ and $\alpha_4$). Note that $h_1$ can be chosen to be $\id$. We will then obtain that $h_1h_2h_3h_4\in \stab([\Sigma_k,\id])$. Note that it will be easier to find an element in $\stab([\Sigma_k,\id])$ that equals to $(h_1h_2h_3h_4)^{-1}$ instead of $(h_1h_2h_3h_4)$. 

\begin{remark}\label{rmk_element_stab_vertices}
By Lemma \ref{lemma_presentation_Stab_Sigma_k}, the stabiliser of a vertex of the form  $[\Sigma_k,\id]$ is a product of a power of the rotation $r_{\Sigma_k}$ and of an element of the braid group. Hence to obtain the power of $r_{\Sigma_k}$ it is enough to understand the images of two adjacent polygons, and to obtain the element of the braid group we need to understand how the punctures inside $\Sigma_k$ are braided. As a consequence, we will follow at each step to what are sent, by $h_1^{-1}$, by $(h_1h_2)^{-1}$, by $(h_1h_2h_3)^{-1}$ and by $(h_1h_2h_3h_4)^{-1}$, the polygons $H_{k+1} $  and $H_r$ as well as how the punctures $\{p_i\}_{1\leq i\leq k}$ are braided.
\end{remark}

\noindent Fix $1\leq i\leq h(n,m)-2$. To shorten the notation let introduce the following braids:
\[\eta_i=\begin{cases}
	\tau_{i+i}\tau_{i} & \text{ if } m\neq i\\
	\tau_{i+1}\tau_{1}\tau_{i} & \text{ if } m=i\\
	\tau_5\tau_2\tau_1\tau_4 & \text{ if } i=4
	
\end{cases}\ \  \text{ and }\ \  \gamma_i=\begin{cases}
\tau_{i}^{-1}& \text{ if } m\neq i\\
\tau_{i}^{-1}\tau_{1}^{-1} & \text{ if } m=i\\
\tau_{4}^{-1}\tau_{1}^{-1}\tau_2^{-1} & \text{ if } i=4
\end{cases}.\] 

\begin{prop}\label{lemma_relation_square_1}
For $m,n\geq 2$, the relations given by squares of $F$ based on a vertex of height $1\leq i\leq h(n,m)-2$ can be chosen as follows: for all  $1+i\leq r\leq i+\left\lceil\frac{m+(n-1)(i-1)-1}{2}\right\rceil$,
	\[
	\gamma_{i}r_{\Sigma_i}^{-r-n+i}\eta_ir_{\Sigma_{i+1}}^{r+n-(i+1)}r_{\Sigma_i}^{i+1-r}=r_{\Sigma_{i-1}}^{i-r}. \]
\end{prop}

\begin{figure}
	\begin{center}
		\includegraphics[scale=0.55]{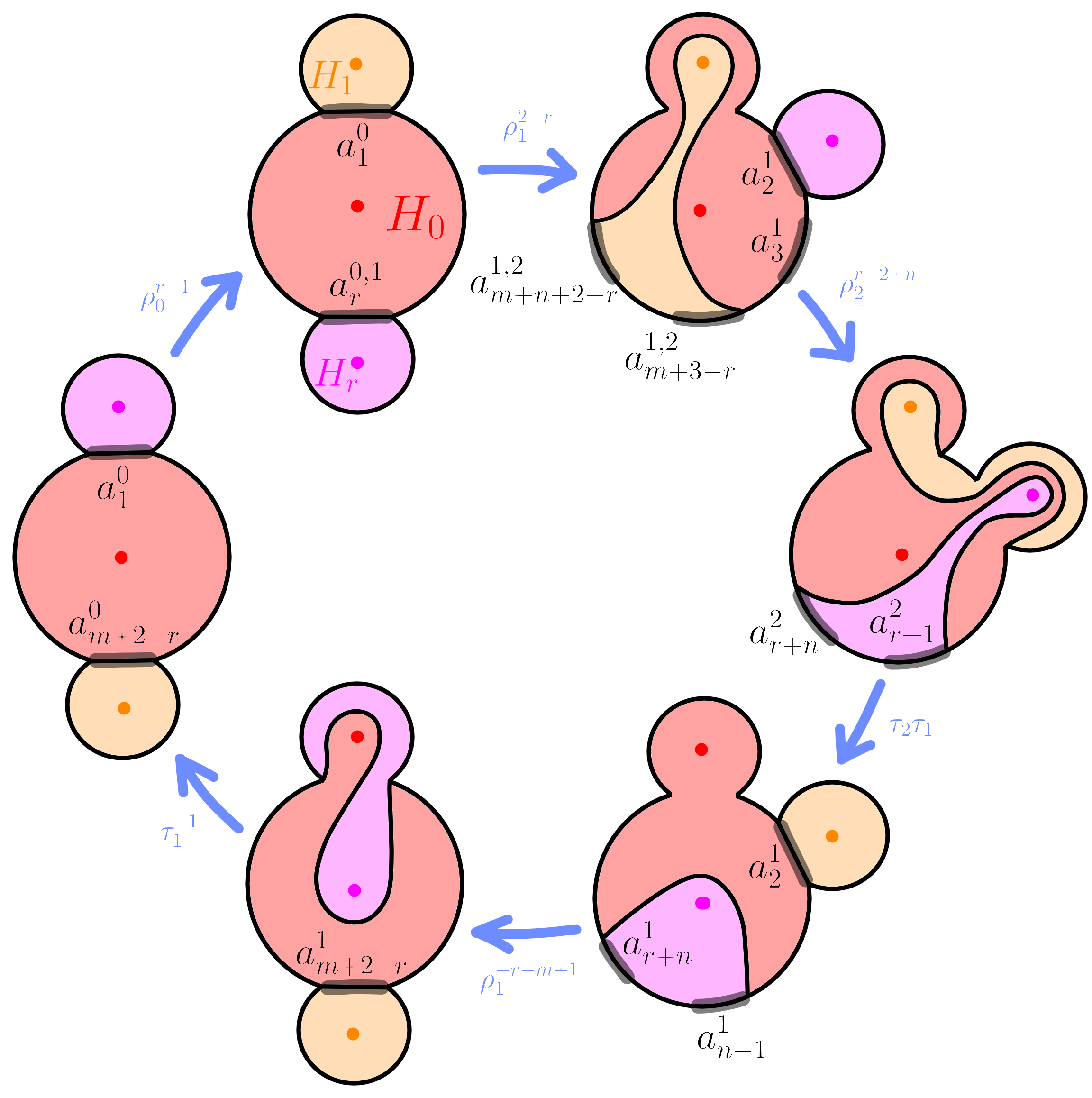}
			\caption{Relation based on a vertex of height $1$. \label{fig:h0}}
	\end{center}
	\end{figure}

\begin{figure}
	\begin{center}
		\includegraphics[scale=0.55]{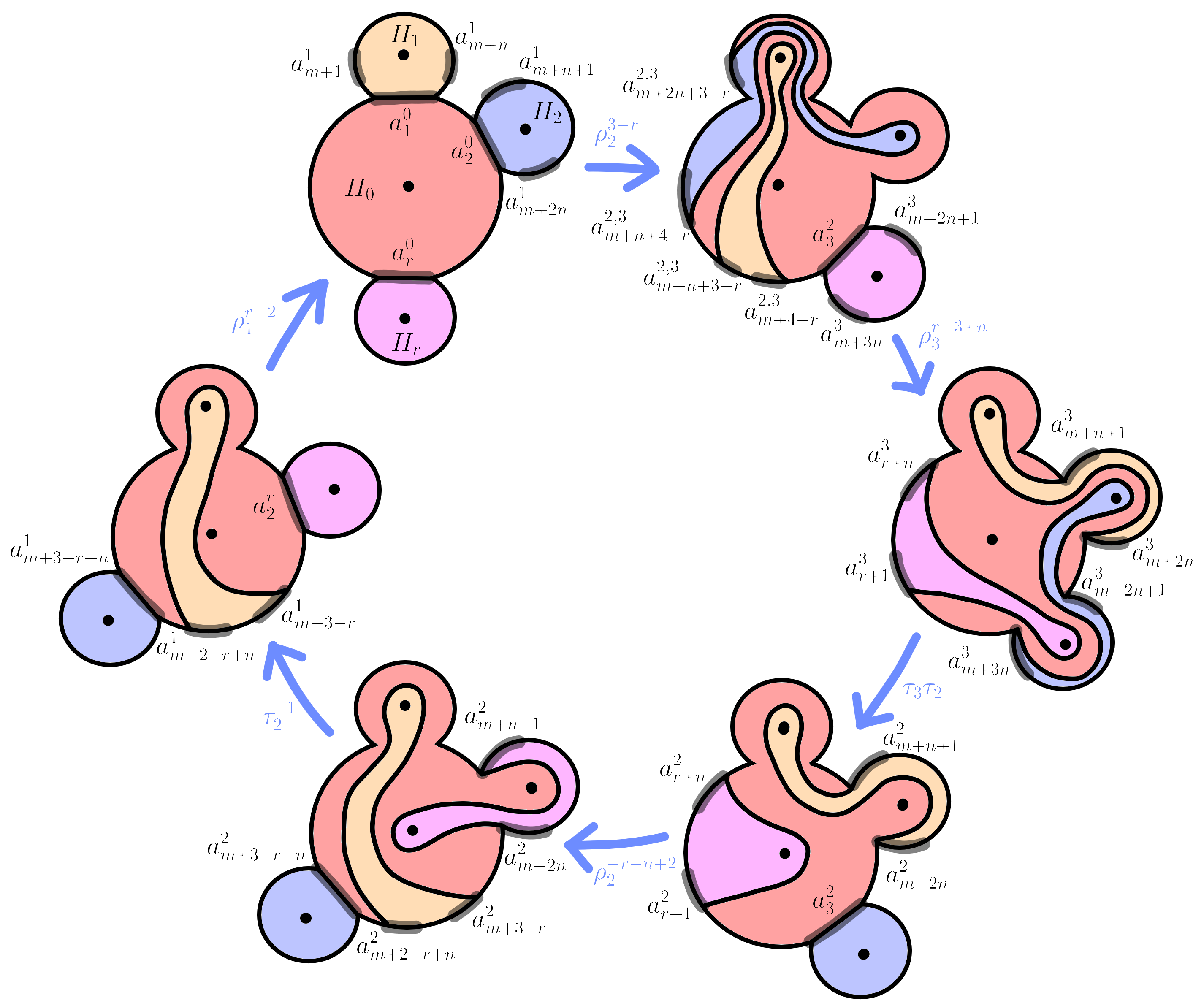}
	\end{center}
	\caption{Relation for $m\geq 3$ based on a vertex of height $2$ for $r\leq m$. \label{fig:h1}}
\end{figure}

	\begin{figure}
	\begin{center}
		\includegraphics[scale=0.50]{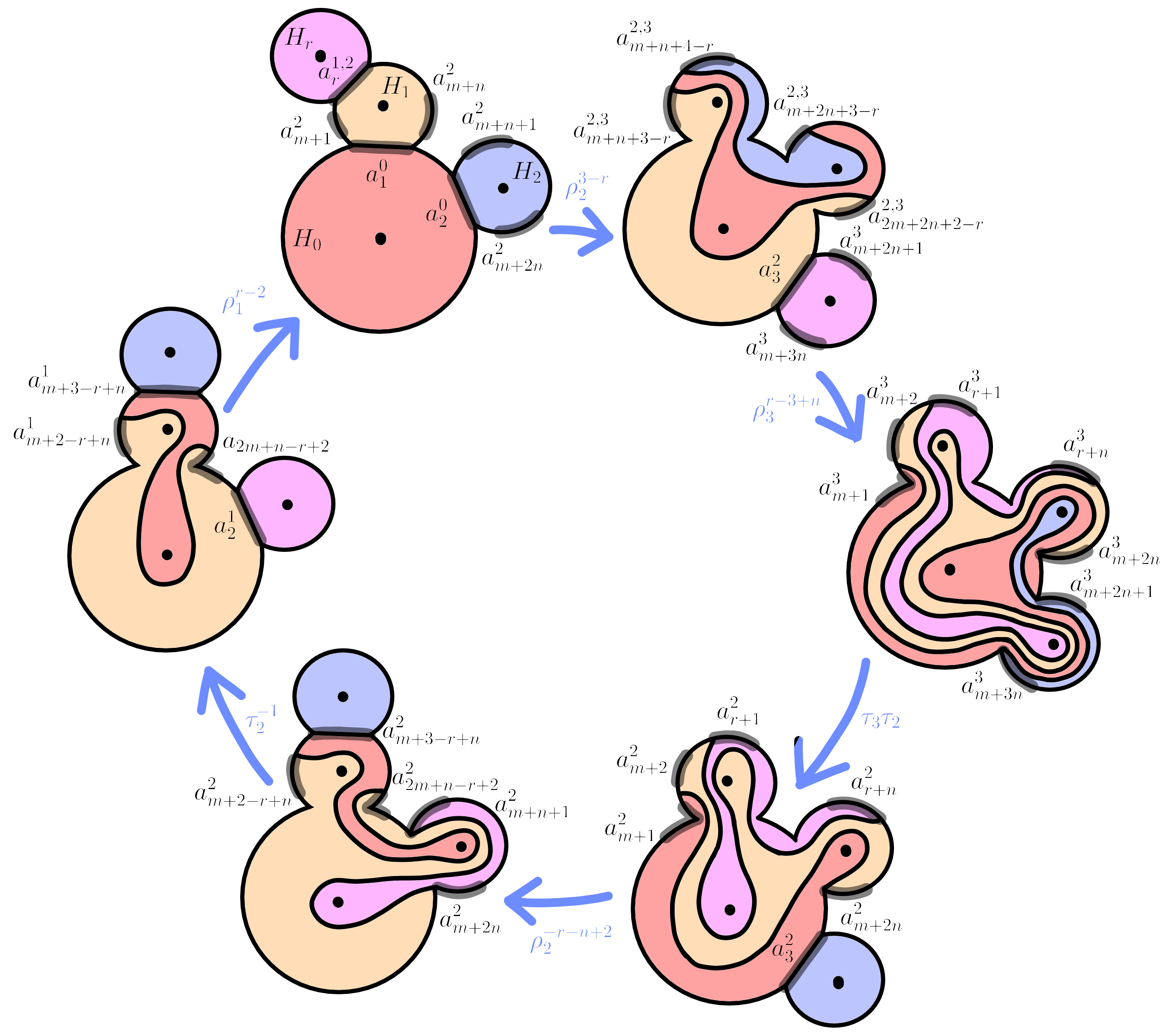}
		
		\caption{Relation for $m\geq 3$ based on a vertex of height $2$ for $r>m$. \label{fig:h1_r}}
	\end{center}
\end{figure}

\begin{figure}
		\begin{center}
			\includegraphics[scale=0.50]{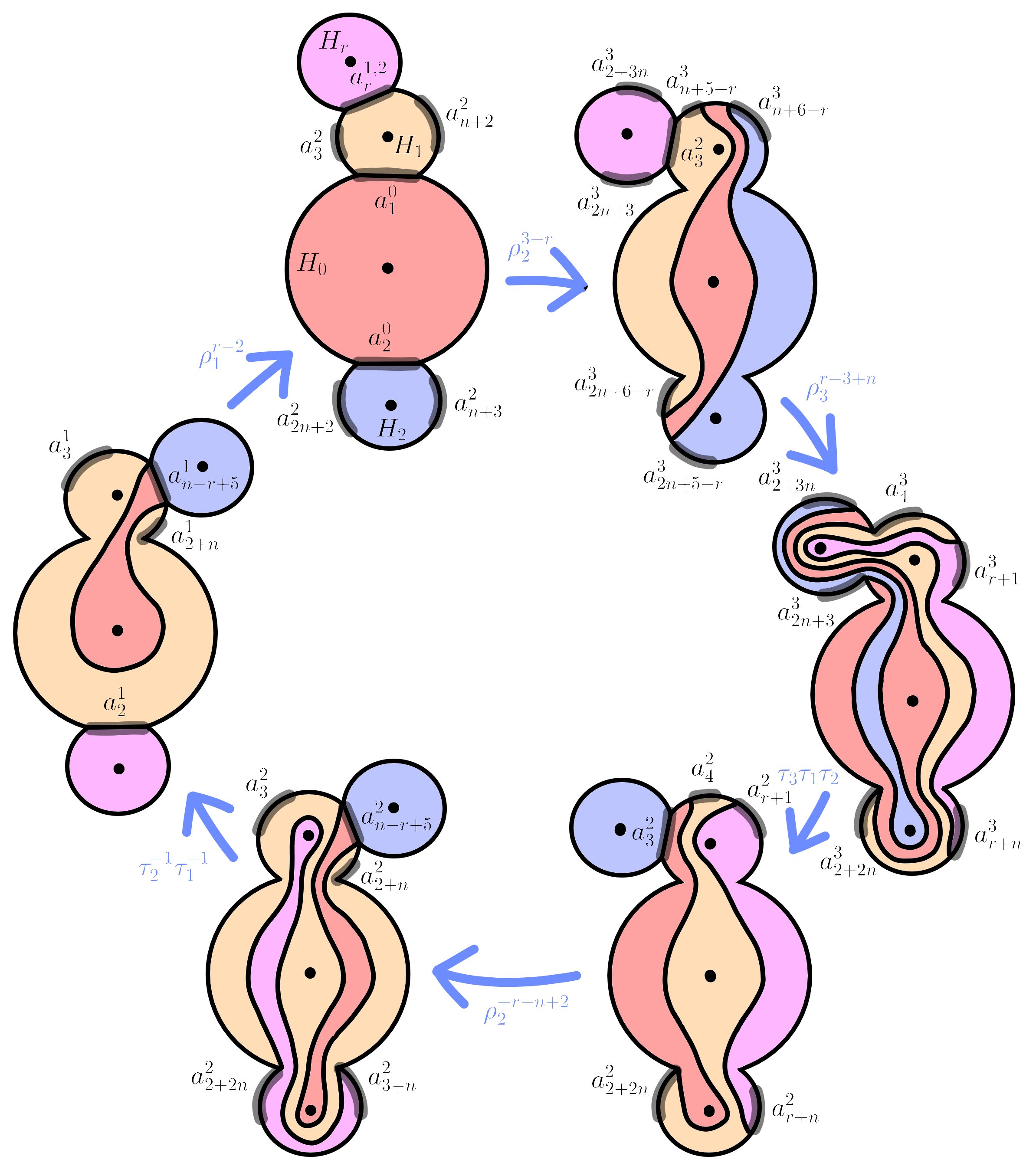}
		\end{center}
	\caption{Relation based on a vertex of height $2$ when $m=2$. \label{fig:h1_m2}}
	\end{figure}

\begin{figure}
		\begin{center}
			\includegraphics[scale=0.50]{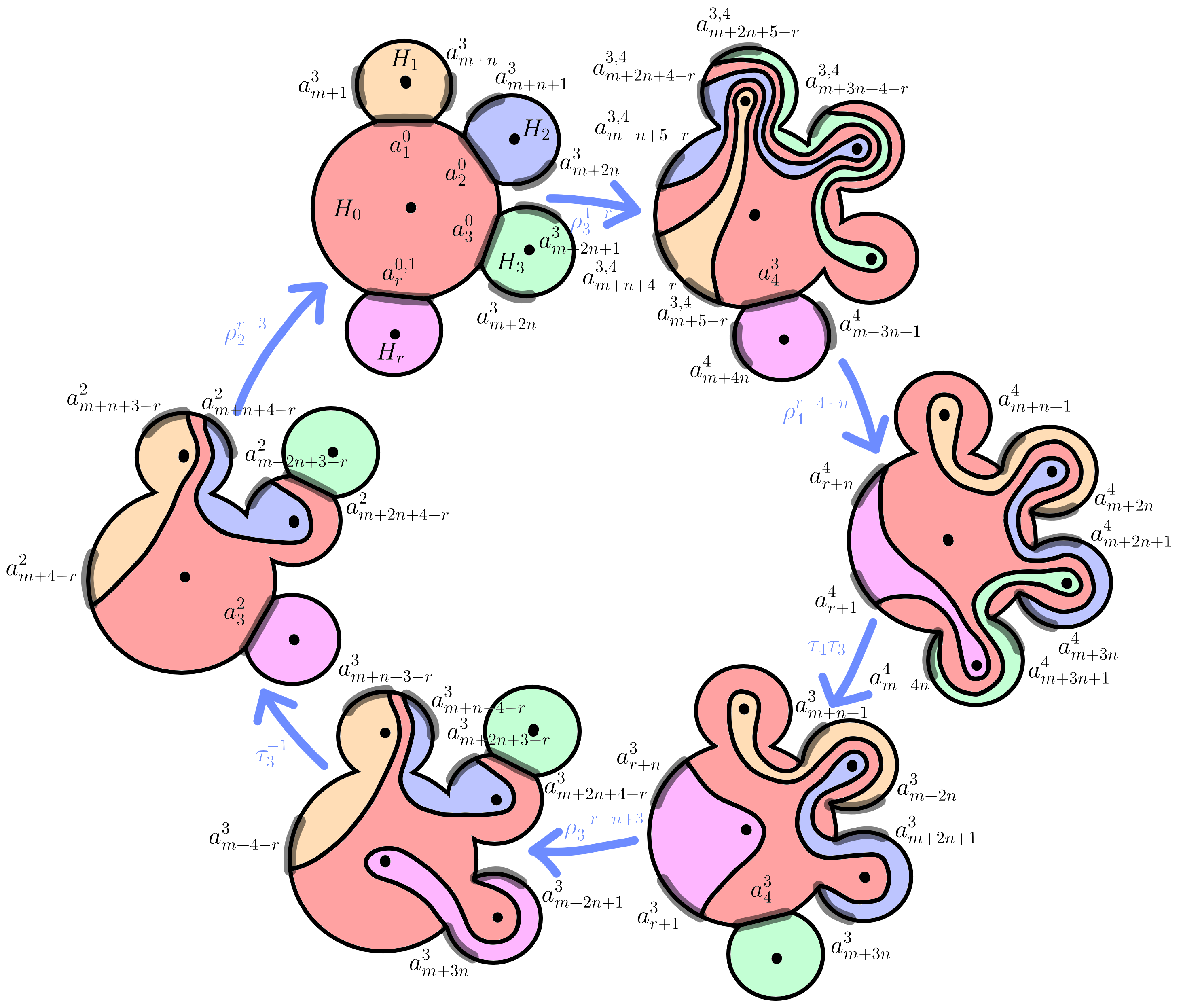}
		\end{center}
	\caption{Relation for $m\geq 3$ based on a vertex of height $3$ for $r\leq m$. \label{fig:h2}}
\end{figure}

\begin{figure}
		\begin{center}
			\includegraphics[scale=0.50]{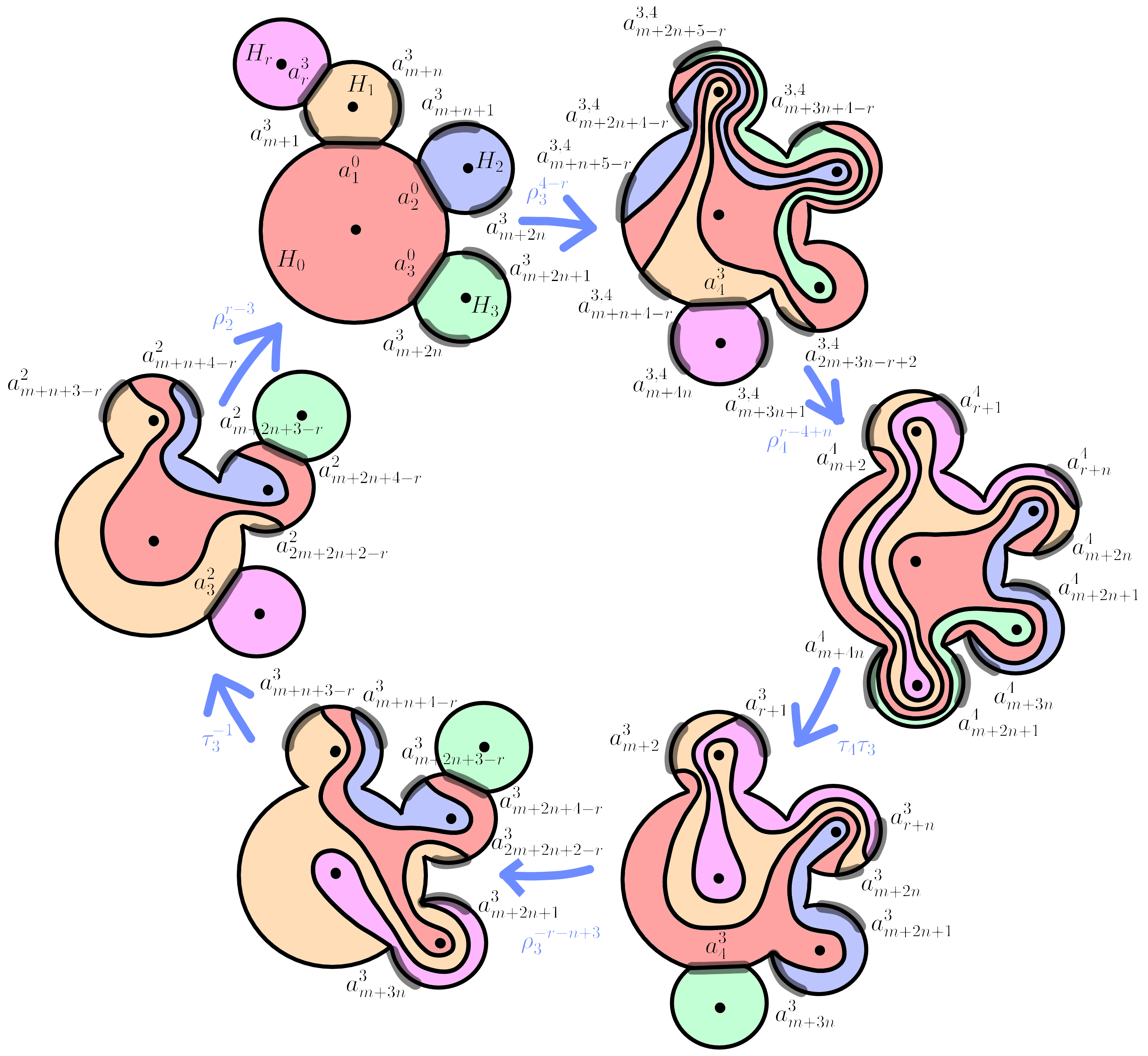}
		\end{center}
	\caption{Relation for $m\geq 3$ based on a vertex of height $3$ for $r>m$. \label{fig:h2_r}}
\end{figure}

\begin{figure}
		\begin{center}
			\includegraphics[scale=0.50]{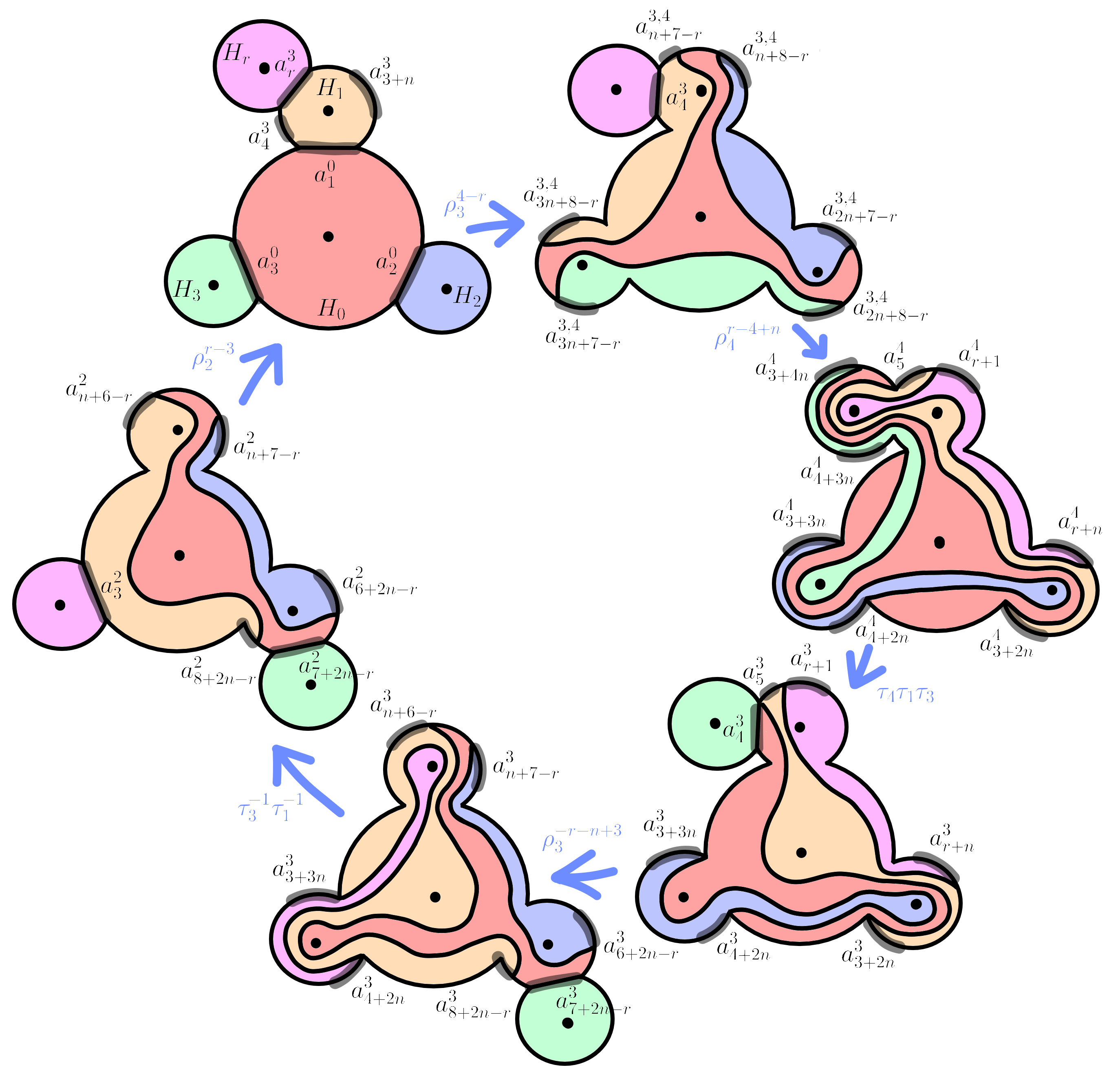}
		\end{center}
	\caption{Relation based on a vertex of height $3$ for $m=3$. \label{fig:h2_m3}}
\end{figure}

\noindent For the completeness of the article the proof is detailed below, but it can be easily read on the figures \ref{fig:h0}, \ref{fig:h1}, \ref{fig:h1_r}, \ref{fig:h1_m2}, \ref{fig:h2}, \ref{fig:h2_r}, \ref{fig:h2_m3}, \ref{fig:h2_m2} and \ref{fig:h4}. By Lemma \ref{lemma_rep_squares}, the squares based on a vertex of height $i$ can be assumed to be of the form of Figure \ref{fig:square} for $k=i-1$ and $1+i\leq r\leq i+\left\lceil\frac{m+(n-1)(i-1)-1}{2}\right\rceil$. Fix $i$ and $r$ as in the statment of Lemma \ref{lemma_relation_square_1}.
%
% \begin{figure}
%	\begin{center}
%		\includegraphics[width=0.4\textwidth]{carre_hauteur_1}
%		\caption{Relations given by squares of height $1$. \label{relation_hauteur_1}}
%	\end{center}
%\end{figure}

 	\begin{figure}
 		\begin{center}
 			\includegraphics[scale=0.40]{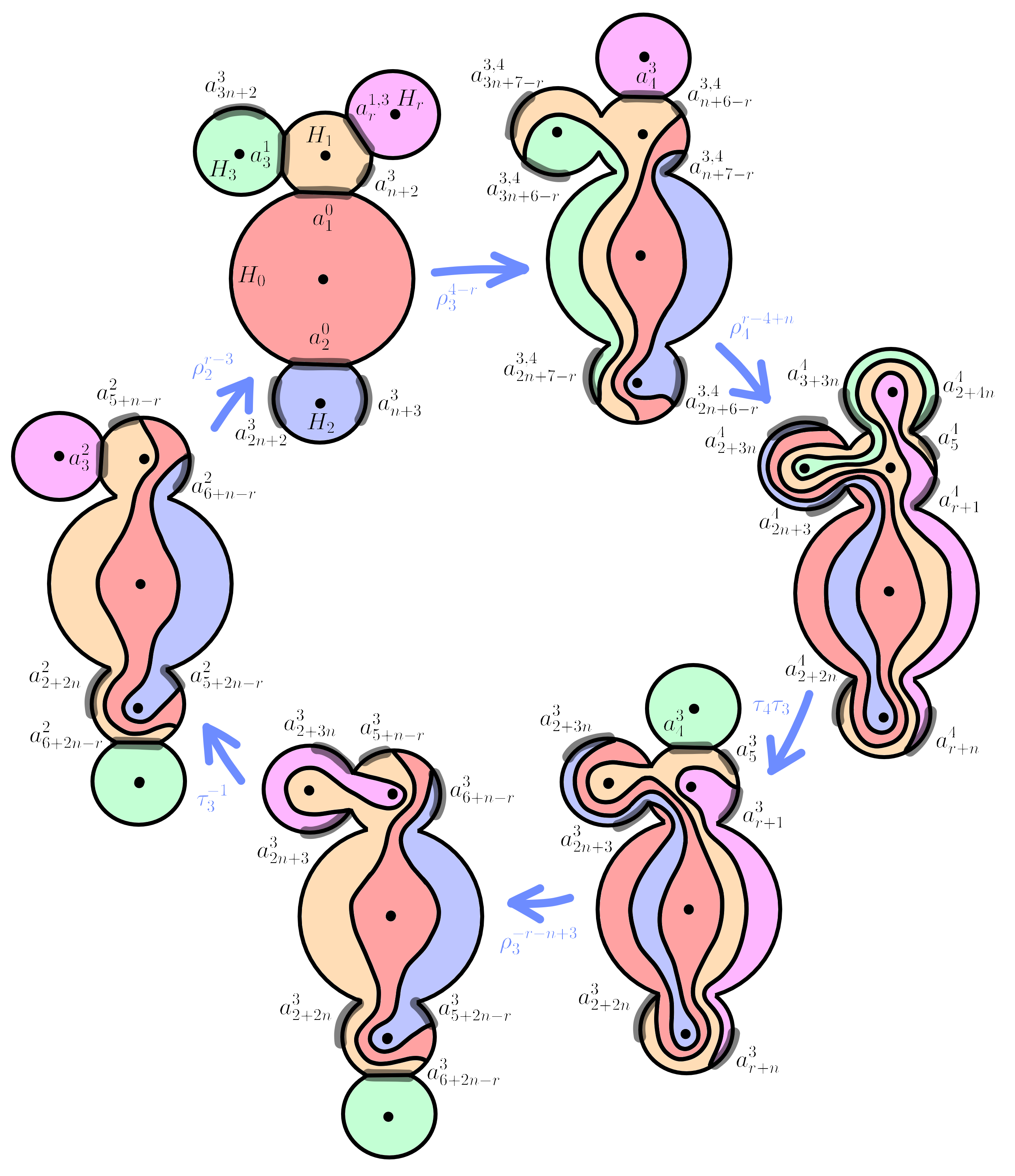}
 		\end{center}
 		\caption{Relation based on a vertex of height $3$ for $m=2$. \label{fig:h2_m2}}
 	\end{figure}
 	
 	\begin{figure}
 		\begin{center}
 			\includegraphics[width=\linewidth]{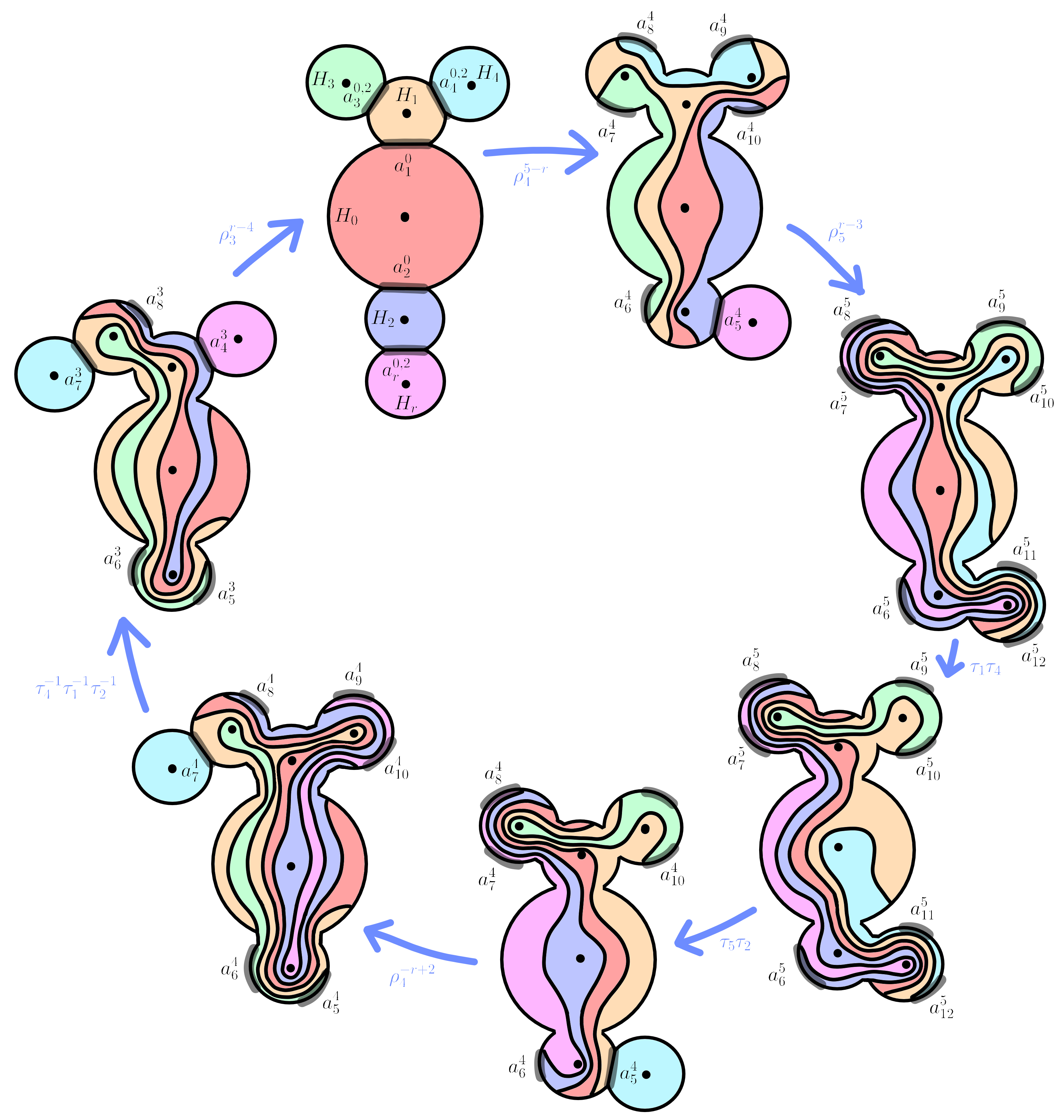}
 		\end{center}
 		\caption{Relation for $m=2$ based on a vertex of height $4$. \label{fig:h4}}
 	\end{figure}
\begin{claim}\label{claim_h1_a_2}
	We can associate the rotation $r_{\Sigma_i}^{r-(i+1)}\in \stab([\Sigma_i,\id])$ to the edge $\alpha_2$.
	\begin{center}
		\begin{tikzcd}
			{[\Sigma_{i-1}\cup H_r,\id] } \arrow[d, "\alpha_4"] & 	{[\Sigma_i\cup H_r,\id] } \arrow[l, "\alpha_3"] \\ 
			{[\Sigma_{i-1},\id] }\arrow[r, "\alpha_1", "\id"' blue ]	& {[\Sigma_i,\id] } \arrow[u,red, "\alpha_2" black]\arrow[r,red, "e_i" black]	 & {[\Sigma_{i+1},\id] } \arrow[ul, dashrightarrow, bend right, red, "r_{\Sigma_i}^{r-(i+1)}" red].\end{tikzcd}
	\end{center}
\end{claim}

\begin{proof}
The edge $\alpha_2$ starts in the vertex $[\Sigma_i,\id]\in V$, and by Fact \ref{fact_edge_facile} $r_{\Sigma_i}^{r-(i+1)}$ belongs to $\stab([\Sigma_i,\id])$ and it sends the edge $e_i$ to the edge $\alpha_2$.
\end{proof}

\begin{claim}\label{claim_h1_a_2_2}
	The rotation $r_{\Sigma_i}^{(i+1)-r}$ sends $H_r$ to $H_{i+1}$ and the braid on the punctures $p_0, \dots, p_i$ is trivial. Moreover, it sends the arcs of $H_i$ to the following arcs of the frontier of $\Sigma_{i+1}$:  \[ \partial_{i+1}r_{\Sigma_i}^{(i+1)-r}(H_i)=\{a_{m+(i-1)n+1+(i+1)-r}^{i+1},\dots, a_{m+in+(i+1)-r}^{i+1}\}.\]
\end{claim}
\begin{proof}
By Fact \ref{fact_edge_facile}, $r_{\Sigma_i}^{(i+1)-r}$ sends the polygon $H_r$ to the polygon $H_{i+1}$. Being a rotation around $\Sigma_i$, the braid induced on $p_0,\dots, p_i$ is trivial.

\noindent Fact~\ref{fact_rotation} gives us: \[\partial_ir_{\Sigma_i}^{(i+1)-r}(H_i)=\{a_{m+(i-1)n+1+(i+1)-r}^{i},\dots, a_{m+in+(i+1)-r}^{i}\}. \]
 Moreover as $m,n\geq 2$ and $r\leq i+\left\lceil\frac{m+(n-1)(i-1)-1}{2}\right\rceil$, we have that:
\[m+(i-1)n+1+(i+1)-r\geq m+(i-1)n-i - \left\lceil\frac{m+n(i-1)-i}{2}\right\rceil+ i+2\geq i+2,\]
and on the other hand, using that $i+1\leq r$ we obtain that
$m+in+(i+1)-r \leq m+in$. Consequently, the indices all belong to $\mathcal{I}_{i}\cap \mathcal{I}_{i+1}$. We conclude using Fact~\ref{fact_arcs_intersection}. \end{proof}

\begin{claim}\label{claim_h1_a_3_2}
	The element $\eta_ir_{\Sigma_{i+1}}^{r+n-(i+1)}r_{\Sigma_i}^{i+1-r}$ sends $H_i$ to $H_{i+1}$, and it sends the arcs of $H_r$ to the following arcs of the frontier of $\Sigma_i$: 
 \[\partial_i\eta_ir_{\Sigma_{i+1}}^{r+n-(i+1)}r_{\Sigma_i}^{i+1-r}(H_r)=\{a_{r+1}^i,\dots, a_{r+n}^i\}.\] Moreover,  \begin{itemize}
\item if $m>i$, it sends $p_r$ to $ p_0$, $p_0$ to $p_i$, $p_i$ to $p_{i+1}$ and fixes $p_1,\dots,p_{i-1}$;
\item if $m=i$, it sends $p_r$ to $p_1$, $p_1$ to $p_0$, $p_0$ to $p_i$, $p_i$ to $p_{i+1}$ and fixes $p_2,\dots,p_{i-1}$;
\item if $m<i\leq 3$, it sends $p_r$ to $p_1$, $p_1$ to $p_i$, $p_i$ to $p_{i+1}$ and fixes $p_0, p_2, \dots, p_{i-1}$.
\item if $i=4$, it sends $p_r$ to $p_2$, $p_2$ to $p_0$, $p_0$ to $p_{1}$, $p_1$ to $p_4$, $p_4$ to $p_5$ and fixes $p_3$.
 \end{itemize} 
\end{claim}

\begin{proof}
The rotation $r_{\Sigma_{i+1}}^{r+n-(i+1)}$ fixes the punctures of the polygons in $\Sigma_{i+1}$ and, applying Fact~\ref{fact_rotation_arc} to Claim \ref{claim_h1_a_2_2}, we have
\[\partial_{i+1}r_{\Sigma_{i+1}}^{r-(i+1)+n}r_{\Sigma_i}^{(i+1)-r}(H_i)=\{a_{m+in+1}^{i+1},\dots, a_{m+(i+1)n}^{i+1}\}=\partial_{i+1}H_{i+1}.\] 
\noindent Applying then $\eta_i$ to unbraid the polygon, we obtain that 
\[	\eta_i r_{\Sigma_{i+1}}^{r-(i+1)+n}r_{\Sigma_i}^{(i+1)-r}(H_i)=H_{i+1}\] and that the punctures are sent as claimed. 

\medskip\noindent By Claim \ref{claim_h1_a_2_2}, $H_{i+1}=r_{\Sigma_i}^{(i+1)-r}(H_r)$. Consequently, by Fact \ref{fact_rotation} and using the equivalence relation $a_{j}^{i+1}=a_{j+m+(i+1)(n-1)}^{i+1}$, we obtain that  \[\partial_{i+1}\eta_ir_{\Sigma_{i+1}}^{r+n-(i+1)}r_{\Sigma_i}^{i+1-r}(H_r)=\{a_{r+1}^{i+1},\dots, a_{r+n}^{i+1}\}.\]  Note that the indices all belong to $\mathcal{I}_{i+1}\cap \mathcal{I}_i$ since $i+1\leq r \leq i+\left\lceil\frac{m+(n-1)(i-1)-1}{2}\right\rceil$, so we conclude using Fact \ref{fact_arcs_intersection}.
\end{proof}

\begin{claim}\label{claim_h1_a_3}
We can associate $(\eta_ir_{\Sigma_{i+1}}^{r-(i+1)+n})^{-1}\in \stab([\Sigma_{i+1},\id])$ to the edge $\alpha_3$.
\begin{center}
	\begin{tikzcd}
		{[\Sigma_{i-1}\cup H_r,\id] } \arrow[d, "\alpha_4"] & 	{[\Sigma_i\cup H_r,\id] } \arrow[l, "\alpha_3"] \\ 
		{[\Sigma_{i-1},\id] }\arrow[r, "\alpha_1", "\id"' blue ]	& {[\Sigma_i,\id] } \arrow[u,"\alpha_2", "r_{\Sigma_i}^{r-(i+1)}"' blue] \arrow[dr, dashrightarrow, bend right, red, "(\eta_ir_{\Sigma_{i+1}}^{r-(i+1)+n})^{-1}"' red]	 & {[\Sigma_{i+1},\id] }\arrow[d,red, "r_{\Sigma_i}^{(i+1)-r}(\alpha_3)" black]\arrow[l,red, "e_1"' black]\\
		&&{[\Sigma_{i-1}\cup H_r,r_{\Sigma_i}^{(i+1)-r}] } \end{tikzcd}
\end{center}
\end{claim}

\begin{proof}
 By Claim \ref{claim_h1_a_2}, $r_{\Sigma_i}^{(i+1)-r}$ sends respectively $o(\alpha_3)=[\Sigma_i\cup H_r,\id]$ and $t(\alpha_3)$ to the vertices $[       , \id]\in V$ and $[\Sigma_{i-1}\cup H_r, r_{\Sigma_i}^{(i+1)-r}]$. Noting that $\eta_ir_{\Sigma_{i+1}}^{r-(i+1)+n}\in \stab([\Sigma_{i+1},\id])$, it remains to prove that $(\eta_ir_{\Sigma_{i+1}}^{r-(i+1)+n})^{-1}$ sends $[\Sigma_i, \id]$ to $[\Sigma_{i-1}\cup H_r, r_{\Sigma_i}^{(i+1)-r}]$.

\medskip\noindent
By Claim \ref{claim_h1_a_2_2}, $\eta_ir_{\Sigma_{i+1}}^{r-(i+1)+n}r_{\Sigma_i}^{(i+1)-r}$ is rigid outside $\Sigma_i\cup H_r$ and it sends $\Sigma_i\cup H_r$ to $\Sigma_{i+1}$. 
By Claim \ref{claim_h1_a_3_2}, $\eta_ir_{\Sigma_{i+1}}^{r-(i+1)+n}r_{\Sigma_i}^{(i+1)-r}$ sends $H_i$ to $H_{i+1}$, hence this application is in fact rigid outside $\Sigma_{i-1}\cup H_r$, and it sends $\Sigma_{i-1}\cup H_r$ to $\Sigma_i$.
Consequently, $(\eta_ir_{\Sigma_{i+1}}^{r-(i+1)+n})^{-1}$ sends
$[\Sigma_i,\id]$ to $[\Sigma_{i-1}\cup H_r, r_{\Sigma_i}^{(i+1)-r}]$ as expected.
\end{proof}

\begin{claim}\label{claim_h1_a_4_2}
	The element $\gamma_{i}r_{\Sigma_i}^{-r-n+i}\eta_ir_{\Sigma_{i+1}}^{r+n-(i+1)}r_{\Sigma_i}^{i+1-r}$ sends $H_i$ to $H_{m+(i-1)n+(i+1)-r}$, $H_r$ to $H_i$, and it induces a trivial braid on the punctures $p_0,\dots, p_{i-1}$. 
\end{claim}

\begin{proof}
Applying Fact \ref{fact_rotation_arc} to Claim \ref{claim_h1_a_3_2} and then using the relation $a_{j}^i=a_ {j+m+(n-1)i}^i$, we obtain
\begin{align*}
\gamma_{i}r_{\Sigma_i}^{-r-n+i}\eta_ir_{\Sigma_{i+1}}^{r+n-(i+1)}r_{\Sigma_i}^{i+1-r}(H_i)&= H_{m+(i-1)n+(i+1)-r},\\
\partial_i\gamma_{i}r_{\Sigma_i}^{-r-n+i}\eta_ir_{\Sigma_{i+1}}^{r+n-(i+1)}r_{\Sigma_i}^{i+1-r}(H_i)(H_r)&=\{a_{m+(i-1)n+1}^i,\dots, a_{m+in}^i\}=\partial H_i.
\end{align*}
Moreover, $\gamma_{i}r_{\Sigma_i}^{-r-n+i}$ exchanges $p_0$ and $p_1$ if $i<m$, it exchanges $p_1$ and $p_i$ if $i>m$, and it sends $p_0$ on $p_1$, $p_1$ on $p_i$ and $p_i$ on $p_1$ so by Claim~\ref{claim_h1_a_3_2}, $\gamma_{i}r_{\Sigma_i}^{-r-n+i}\eta_ir_{\Sigma_{i+1}}^{r+n-(i+1)}r_{\Sigma_i}^{i+1-r}(H_i)$ induces a trivial braid on the punctures $p_0,\dots, p_{i-1}$ and it sends $p_r$ to $p_i$.
 Consequently  $\gamma_{i}r_{\Sigma_i}^{-r-n+i}\eta_ir_{\Sigma_{i+1}}^{r+n-(i+1)}r_{\Sigma_i}^{i+1-r}(H_r)=H_i$.
 \end{proof}

\begin{claim}\label{claim_h1_a_4}
	We can associate $(\gamma_{i}r_{\Sigma_i}^{-r-n+i})^{-1}\in \stab([\Sigma_i,\id])$ to the edge $\alpha_4$.
	\begin{center}
		\begin{tikzcd}
			{[\Sigma_{i-1}\cup H_r,\id] } \arrow[d, "\alpha_4"] & 	{[\Sigma_i\cup H_r,\id] } \arrow[l, "(\eta_ir_{\Sigma_{i+1}}^{r+n-(i+1)})^{-1}"' blue, "\alpha_3"] \\ 
			{[\Sigma_{i-1},\id] }\arrow[r, "\id"' blue, "\alpha_1" ]\arrow[rd, dashrightarrow, bend right, red, "(\gamma_{i}r_{\Sigma_i}^{-r-n+i})^{-1}"' red]	& {[\Sigma_i,\id] } \arrow[u,"r_{\Sigma_i}^{r-(i+1)}"' blue, "\alpha_2"]\arrow[d,red,"\eta_ir_{\Sigma_{i+1}}^{r+n-(i+1)}r_{\Sigma_i}^{i+1-r}(\alpha_4)" black]\arrow[l,red,shift left=1.5ex]\\
			&{[\Sigma_{i-1},\eta_ir_{\Sigma_{i+1}}^{r+n-(i+1)}r_{\Sigma_i}^{i+1-r}] } \end{tikzcd}
	\end{center}
\end{claim}

\begin{proof}
By Claim \ref{claim_h1_a_3}, $\eta_ir_{\Sigma_{i+1}}^{r+n-(i+1)}r_{\Sigma_i}^{i+1-r}$ sends respectively $o(\alpha_4)=[\Sigma_{i-1}\cup H_r,\id]$ and $t(\alpha_4)$ to the vertices $[\Sigma_i, \id]\in V$ and $[\Sigma_{i-1}, \eta_ir_{\Sigma_{i+1}}^{r+n-(i+1)}r_{\Sigma_i}^{i+1-r}]$. 
Noting that $\gamma_{i}r_{\Sigma_i}^{-r-n+i}\in \stab([\Sigma_i,\id])$, it remains to prove that $(\gamma_{i}r_{\Sigma_i}^{-r-n+i})^{-1}$ sends $[\Sigma_{i-1}, \id]$ to $[\Sigma_{i-1}, \eta_ir_{\Sigma_{i+1}}^{r+n-(i+1)}r_{\Sigma_i}^{i+1-r}]$.

\medskip\noindent
As just seen $\gamma_{i}r_{\Sigma_i}^{-r-n+i}\eta_ir_{\Sigma_{i+1}}^{r+n-(i+1)}r_{\Sigma_i}^{i+1-r}$ is rigid outside $\Sigma_{i-1}\cup H_r$, and it sends $\Sigma_{i-1}\cup H_r$ to $\Sigma_i$. By Claim \ref{claim_h1_a_4_2}, $\gamma_{i}r_{\Sigma_i}^{-r-n+i}\eta_ir_{\Sigma_{i+1}}^{r+n-(i+1)}r_{\Sigma_i}^{i+1-r}$ sends $H_r$ to $H_i$, hence this application is in fact rigid outside $\Sigma_{i-1}$ and preserves it. Consequently, $(\gamma_{i}r_{\Sigma_i}^{-r-n+i})^{-1}$ sends $[\Sigma_{i-1} \id]$ to $[\Sigma_{i-1}, \eta_ir_{\Sigma_{i+1}}^{r+n-(i+1)}r_{\Sigma_i}^{i+1-r}]$ as expected.
\end{proof}

\begin{proof}[Proof of Lemma \ref{lemma_relation_square_1}]	Fix $1\leq i\leq h(n,m)=(2,2)$ and $1+i\leq r\leq i+\left\lceil\frac{m+(n-1)(i-1)-1}{2}\right\rceil$. 
	Using Brown's method, we obtain by Claims \ref{claim_h1_a_2}, \ref{claim_h1_a_3} and \ref{claim_h1_a_4} that $\gamma_{i}r_{\Sigma_i}^{-r-n+i}\eta_ir_{\Sigma_{i+1}}^{r+n-(i+1)}r_{\Sigma_i}^{i+1-r}$ belongs to $\stab\left([\Sigma_{i-1}, \id]\right)$:
	\begin{center}
		\begin{tikzcd}
			{[\Sigma_{i-1}\cup H_r,\id] } \arrow[d,"(\gamma_{i}r_{\Sigma_i}^{-r-n+i})^{-1}"' blue ] & 		{[\Sigma_i\cup H_r,\id] } \arrow[l, "(\eta_ir_{\Sigma_{i+1}}^{r+n-(i+1)})^{-1}"' blue] \\ 
			{[\Sigma_{i-1},\id] }\arrow[r, "\id" blue ]	& {[\Sigma_i,\id] }\arrow[u,"r_{\Sigma_i}^{r-(i+1)}"' blue] \end{tikzcd}
	\end{center}
	and, by Claim \ref{claim_h1_a_4_2} that it sends $H_i$ to $H_{m+(i-1)n+(i+1)-r}$, $H_r$ to $H_i$ and it induces a trivial braid on the punctures $p_0,\dots, p_{i-1}$.
	This implies, by Remark~\ref{rmk_element_stab_vertices}, the relation announced.
\end{proof}

\subsection{Presentation of the braided Higman-Thompson groups for $n,m\geq 2$}
\noindent Applying \cite[Theorem 1]{Brown_presentation}, and using Lemmas \ref{lemma_presentation_Stab_Sigma_k}, \ref{lemma_stab_edge} and \ref{lemma_relation_square_1}, we obtain the following presentation of the braided Higman-Thompson groups.

\begin{thm}\label{thm_presentation_BHT} For $n,m\geq 2$, the group $\mathrm{br}T_{n,m}$ is generated by $\{r_{\Sigma_k}\}_{0\leq k \leq \bar{h}(n,m)}$ and by $\{\tau_{k}\}_{1\leq k\leq \bar{h}(n,m)}$, where $\bar{h}(2,2)=5$ and $\bar{h}(n,m)=4$ otherwise. The relations are generated by:
	\begin{itemize}
			\item \emph{the braids relations: } 
			\begin{enumerate}
					\item when $m<4$, $\tau_i\tau_\ell=\tau_\ell\tau_i$, for any $2\leq i\leq m<\ell\leq 4$,
					\item $\tau_i\tau_j\tau_i=\tau_j\tau_i\tau_j$, for any $1\leq i<j\leq \min(4,m)$,
					\item\label{item:rel_braid_3} when $m<4$, $\tau_1\tau_\ell\tau_1=\tau_\ell\tau_1\tau_\ell$, for any $m<\ell \leq 4$,
					\item $\tau_i\tau_j\tau_s\tau_i=\tau_j\tau_s\tau_i\tau_j=\tau_s\tau_i\tau_j\tau_s$, for any $1\leq i<j<s\leq \min(4,m)$,
					\item when $m=2$, $\tau_{3}\tau_4\tau_{3}=\tau_4\tau_{3}\tau_4$ and  $\tau_1\tau_{3}\tau_4\tau_1=\tau_{3}\tau_4\tau_1\tau_{3}=\tau_4\tau_1\tau_{3}\tau_4$,
					\item when $(n,m)=(2,2)$, $\tau_5\tau_i=\tau_i\tau_5$, for any $i\in\{1,3,4\}$ and $\tau_2\tau_5\tau_2=\tau_5\tau_2\tau_5$.
				\end{enumerate}
			\item \emph{the commutation relations:} $r_{\Sigma_k}\tau_i=\tau_ir_{\Sigma_k}$ for $1\leq i\leq k\leq \bar{h}(n,m)$,
			\item \emph{the rotation relations:} $r_{\Sigma_k}^{m+k(n-1)}=(\tau_k\tau_{k-1}\dots\tau_1)^{-(k+1)}$ for $0\leq k\leq \bar{h}(n,m)$.
			\item \emph{the square relations}: for $1\leq i\leq \bar{h}(n,m)-1$ and $1\leq j_i \leq \left\lceil\frac{m+(n-1)(i-1)-1}{2}\right\rceil$
		
			\[	r_{\Sigma_{i-1}}^{j_i}\gamma_ir_{\Sigma_i}^{-n-j_i}\eta_{i}r_{\Sigma_{i+1}}^{j_i+n-1}r_{\Sigma_i}^{1-j_i}=\id ,\]
	
			where
			\[\eta_i=\begin{cases}
					\tau_{i+i}\tau_{i} & \text{ if } m\neq i\\
					\tau_{i+1}\tau_{1}\tau_{i} & \text{ if } m=i\\
					\tau_5\tau_2\tau_1\tau_4 & \text{ if } i=4
					
				\end{cases}\ \  \text{ and }\ \  \gamma_i=\begin{cases}
					\tau_{i}^{-1}& \text{ if } m\neq i\\
					\tau_{i}^{-1}\tau_{1}^{-1} & \text{ if } m=i\\
					\tau_{4}^{-1}\tau_{1}^{-1}\tau_2^{-1} & \text{ if } i=4
				\end{cases}.\] 
		\end{itemize}
\end{thm}

\noindent As a direct corollary, we obtain a new presentation of the Higman-Thompson groups $T_{n,m}$.
\begin{cor}
	For $n,m\geq 2$, $T_{n,m}$ is generated by $\{r_{\Sigma_0}\}_{\bar{h}(n,m)}$, where $\bar{h}(2,2)=5$ and $\bar{h}(n,m)=4$ otherwise. The relations are generated by:
	\begin{itemize}
			\item \emph{the rotation relations:} $r_{\Sigma_k}^{m+k(n-1)}=\id$ for $0\leq k\leq \bar{h}(n,m)$.
			\item \emph{the square relations}: for $1\leq i\leq \bar{h}(n,m)-1$ and $1\leq j_i \leq \left\lceil\frac{m+(n-1)(i-1)-1}{2}\right\rceil$
			
			\[	r_{\Sigma_{i-1}}^{j_i}r_{\Sigma_i}^{-n-j_i}r_{\Sigma_{i+1}}^{j_i+n-1}r_{\Sigma_i}^{1-j_i}=\id.\]
		\end{itemize}
\end{cor}

	\section{Abelianisation}\label{section_abelianisation}
	\noindent In this section we compute the abelianisation of $\mathrm{br}T_{n,m}$ for $m,n\geq 2$ in order to obtain some restriction for the isomorphism problem of Section \ref{section_isom_problem}.

\begin{proof}[Proof of Theorem \ref{thm_abelianise}]
	In order to compute the abelianisation of $\mathrm{br}T_{n,m}$, we look at the presentation of $\mathrm{br}T_{n,m}$ that we computed in Theorem \ref{thm_presentation_BHT} and we deduce from it a presentation of the abelianisation.
	
	\medskip\noindent By Theorem \ref{thm_presentation_BHT}, $\mathrm{br}T_{n,m}$ is generated by $r_{\Sigma_i}$ for $0\leq i\leq \bar{h}(n,m)$ and by $\tau_j$ for $1\leq j\leq\bar{h}(n,m)$, hence the abelianisation is generated by their class $\bar{r}_{\Sigma_i}$ and $\bar{\tau_i}$. Moreover the relations of Theorem \ref{thm_presentation_BHT} gives us the following relations in the quotient.
	\begin{itemize}
\item The braided relations gives us that $\bar{\tau_1}=\bar{\tau_i}$ for $2\leq i\leq \bar{h}(n,m)$. So, in what follows we will rename it $\bar{\tau}$.
\item The rotation relations gives us: for $0\leq k \leq \bar{h}(n,m)$, $\bar{r}_{\Sigma_k}^{m+k(n-1)}=\bar{\tau}^{-k(k+1)}$.

\item The square relations gives us: for $1\leq i\leq \bar{h}(n,m)-1$ and $1\leq j_i \leq \left\lceil\frac{m+(n-1)(i-1)-1}{2}\right\rceil$

 \[\bar{\tau}\bar{r}_{\Sigma_{i-1}}^{j_i}\bar{r}_{\Sigma_i}^{-n-2j_i+1}\bar{r}_{\Sigma_{i+1}}^{j_i+n-1}=\bar{\id}. \]
 
\noindent More explicitly, this gives us the following relations:
\begin{enumerate}
	\item\label{carre_1} for $i=1$, $\bar{\tau}\bar{r}_{\Sigma_0}\bar{r}_{\Sigma_1}^{-1-n}\bar{r}_{\Sigma_2}^n=\bar{\id}$ and, when $m\geq 4$: $\bar{r}_{\Sigma_0}\bar{r}_{\Sigma_1}^{-2}\bar{r}_{\Sigma_2}=\bar{\id}$,
	\item \label{carre_2}for $i=2$, $\bar{\tau}\bar{r}_{\Sigma_1}\bar{r}_{\Sigma_2}^{-1-n}\bar{r}_{\Sigma_3}^n=\bar{\id}$ and, when $(n,m)\neq (2,2)$: $\bar{r}_{\Sigma_1}\bar{r}_{\Sigma_2}^{-2}\bar{r}_{\Sigma_3}=\bar{\id}$,
	\item \label{carre_3}for $i=3$, $\bar{\tau}\bar{r}_{\Sigma_2}\bar{r}_{\Sigma_3}^{-1-n}\bar{r}_{\Sigma_4}^n=\bar{\id}$ and  $\bar{r}_{\Sigma_2}\bar{r}_{\Sigma_3}^{-2}\bar{r}_{\Sigma_4}=\bar{\id}$.
		\item \label{carre_4}for $i=4$ (only when $(n,m)=(2,2)$), $\bar{\tau}\bar{r}_{\Sigma_3}\bar{r}_{\Sigma_4}^{-1-n}\bar{r}_{\Sigma_5}^n=\bar{\id}$ and  $\bar{r}_{\Sigma_3}\bar{r}_{\Sigma_4}^{-2}\bar{r}_{\Sigma_5}=\bar{\id}$.
\end{enumerate}
	\end{itemize}
	
\noindent For any $m,n\geq 2$, the square relations \ref{carre_1}, \ref{carre_2}, \ref{carre_3} and \ref{carre_4} are equivalent to:
\[\left\{ \begin{array}{l} 
		\bar{\tau} =\bar{r}_{\Sigma_3}^{n-1} \bar{r}_{\Sigma_4}^{-(n-1)}\\
			\bar{r}_{\Sigma_5}= \bar{r}_{\Sigma_3}^{-1} \bar{r}_{\Sigma_4}^{2}\\
	\bar{r}_{\Sigma_2}= \bar{r}_{\Sigma_3}^2 \bar{r}_{\Sigma_4}^{-1}\\
	\bar{r}_{\Sigma_1}= \bar{r}_{\Sigma_3}^3 \bar{r}_{\Sigma_4}^{-2} \\
	\bar{r}_{\Sigma_0}= \bar{r}_{\Sigma_3}^4 \bar{r}_{\Sigma_4}^{-3}\\
\end{array} \right. \]
and so the abelianisation is generated by $\bar{r}_{\Sigma_3}$ and $\bar{r}_{\Sigma_4}$.
Plugging them in the rotation relations we obtain:
\begin{enumerate}[\alph*.]
\item\label{carre_bis0} $\bar{r}_{\Sigma_3}^{4m}= \bar{r}_{\Sigma_4}^{3m}$,
\item\label{carre_bis1} $\bar{r}_{\Sigma_3}^{3m+5(n-1)}= \bar{r}_{\Sigma_4}^{2m+4(n-1)}$,
\item\label{carre_bis2} $\bar{r}_{\Sigma_3}^{2m+10(n-1)}= \bar{r}_{\Sigma_4}^{m+8(n-1)}$,
\item\label{carre_bis3} $\bar{r}_{\Sigma_3}^{m+15(n-1)}= \bar{r}_{\Sigma_4}^{12(n-1)}$,
\item\label{carre_bis4} $\bar{r}_{\Sigma_3}^{20(n-1)}= \bar{r}_{\Sigma_4}^{-m+16(n-1)}$
\item\label{carre_bis5} when $(n,m)=(2,2)$,  $\bar{r}_{\Sigma_3}^{23}= \bar{r}_{\Sigma_4}^{16}$.
\end{enumerate}

\noindent When $(n,m)=(2,2)$, using \eqref{carre_bis1} and \eqref{carre_bis5}, we obtain that $\bar{r}_{\Sigma_3}=\bar{\id}$. Hence we have that the abelianisation of $\mathrm{br}T_{2,2}$ is isomorphic to $\Z_2$ and is generated by the class of $\bar{r}_{\Sigma_4}=\bar{r}_{\Sigma_0}$.

\medskip\noindent Now we focus on the case $(n,m)\neq (2,2)$.
 As \eqref{carre_bis4} is equal (in additive notation) to \eqref{carre_bis0}-(\eqref{carre_bis1}+\eqref{carre_bis3}), \eqref{carre_bis3} is equal to \eqref{carre_bis0}-(\eqref{carre_bis1}+\eqref{carre_bis2}) and \eqref{carre_bis2} is equal to 2\eqref{carre_bis1}-\eqref{carre_bis0}, these rotation relations are equivalent to:\[\left\{ \begin{array}{l} 
\bar{r}_{\Sigma_3}^{4m}= \bar{r}_{\Sigma_4}^{3m}\\
\bar{r}_{\Sigma_3}^{3m+5(n-1)}= \bar{r}_{\Sigma_4}^{2m+4(n-1)}\end{array} \right..\]

\noindent Hence, we obtain that the abelianisation is generated by:
\begin{align*}
\langle 
\bar{r}_{\Sigma_3},\bar{r}_{\Sigma_4} \mid \bar{r}_{\Sigma_3}^{4m}\bar{r}_{\Sigma_4}^{-3m}= \bar{\id}, &\ \bar{r}_{\Sigma_3}^{3m+5(n-1)}\bar{r}_{\Sigma_4}^{-2m-4(n-1)} = \bar{\id}\rangle\\
&\underset{t=\bar{r}_{\Sigma_3}\bar{r}_{\Sigma_4}^{-1}}{=}\langle 
t,\bar{r}_{\Sigma_4} \mid t^{4m}\bar{r}_{\Sigma_4}^{m}=\bar{\id}, \ t^{3m+5(n-1)} \bar{r}_{\Sigma_4}^{m+(n-1)}=\bar{\id} \rangle\\
&\underset{\bar{r}_{\Sigma_0}=t^4\bar{r}_{\Sigma_4}}{=}\langle 
t,\bar{r}_{\Sigma_0}\mid \bar{r}_{\Sigma_0}^{m}
=\bar{\id}, t^{-m+n-1} \bar{r}_{\Sigma_0}^{m+(n-1)}=\bar{\id} \rangle\\
&\underset{v=t\bar{r}_{\Sigma_0}}{=}\langle 
v,\bar{r}_{\Sigma_0}\mid \bar{r}_{\Sigma_0}^{m}
=\bar{\id}, v^{-m+n-1}=\bar{\id} \rangle.
\end{align*}
and so that the abelianisation of $\mathrm{br}T_{n,m}$ is isomorphic to $\Z_m\times \Z_{\lvert m-n+1\rvert }$ as expected.
\end{proof}

\begin{remark}\label{rmk_abel_tnr}
As explained in \cite{GLU_finiteness}, there exists the following short exact sequence:
\[1 \to B_\infty \to \mathrm{br}T_{n,m} \to T_{n,m} \to 1.\]
As a consequence, killing $\bar{\tau}$ in the computation of the abelianisation of $\mathrm{br}T_{n,m}$ allows us to recover a presentation of the abelianisation of the Brown-Thompson groups $T_{n,m}$ (\cite{Brown_finiteness_properties}) which is isomorphic to $\Z_{\gcd(m,n-1)} \times\Z_{\gcd(m,n-1)}$. More precisely, we have to add the relation $\bar{r}_{\Sigma_3}^{n-1}= \bar{r}_{\Sigma_4}^{n-1}$, which allows us to obtain the following presentation \[\langle 
t,\bar{r}_{\Sigma_0}\mid \bar{r}_{\Sigma_0}^{m}
=\bar{\id}, t^{n-1}=\bar{\id}, t^{-m} \bar{r}_{\Sigma_0}^{(n-1)}=\bar{\id} \rangle.\] The result can be deduced by putting for instance the last relation to the power $\frac{m}{\gcd(m,n-1)}$.
\end{remark}

\section{Isomorphism problem}\label{section_isom_problem}
\noindent
This section is dedicated to the proof of the partial result on the isomorphism problem of the braided Higman-Thompson groups $\mathrm{br}T_{n,m}$ given by Theorem \ref{thm:BigIntro}.

\medskip\noindent
The strategy is the following. As a consequence of the algebraic characterisation of the subgroup $B_\infty$ given by Theorem~\ref{thm:BraidSubCharacteristic}, an isomorphism $\mathrm{br}T_{n,m} \to \mathrm{br}T_{r,s}$ induces an isomorphism $T_{n,m} \to T_{r,s}$. By a standard argument based on Rubin's theorem, we deduce that $r=n$. Next, we deduce the equality $s=m$ or $s=\lvert m-n+1\rvert $ from our previous computation of abelianisations. 

\begin{thm}\label{thm:BraidSubCharacteristic}
Let $n,m \geq 2$ be integers. The subgroup $B_\infty$ of $\mathrm{br}T_{n,m}$ is the unique subgroup that is maximal (with respect to the inclusion) among the subgroups $N$ satisfying the property
\begin{itemize}
	\item[($\ast$)] $N$ is normal and $\mathrm{br}T_{n,m}/N$ does not surject onto a virtually abelian group with a kernel that has a non-trivial centre. 
\end{itemize}
\end{thm}

\noindent
Recall from \cite{Brown_finiteness_properties} that there exists a morphism $\theta : T_{n,m} \twoheadrightarrow \mathbb{Z}/d\mathbb{Z}$, where $d:= \mathrm{gcd}(m,n-1)$, such that the commutator subgroup $T_{n,m}^s$ of the finite-index subgroup $T_{n,m}^0:= \mathrm{ker}(\theta)$ is simple. More precisely, $\theta$ is defined as follows. Given an element $g \in T_{n,m}$, we represent it as a triple $(R,\sigma,S)$, where $R$ and $S$ are two finite binary rooted trees with the same number of leaves and where $\sigma$ is a bijection from the leaves of $R$ to the leaves of $S$. A requirement is that $\sigma$ preserves the ``cyclic orderings'' on the leaves of $R$ and $S$. Namely, we think of the leaves of $R$ and $S$ as numbered from left to right modulo $N$, the total number of leaves, and $\sigma$ then sends each leaf numbered $i$ to the leaf numbered $i+k$ for some fixed $k$. Then $\theta(g)$ is defined by taking $k$ mod $d$. 

\begin{prop}[\cite{Brown_finiteness_properties}]\label{prop:TAlmostSimple}
Let $n,m \geq 2$ be integers. Every non-trivial normal subgroup of $T_{n,m}$ contains $T_{n,m}^s$. 
\end{prop}

\noindent We first verify that:

\begin{lemma}\label{lem:BhasAst}
The subgroup $B_\infty$ of $\mathrm{br}T_{n,m}$ satisfies the property ($\ast$).
\end{lemma}

\begin{proof}
In other words, we want to prove that $T_{n,m}$ does not surject onto a virtually abelian group with a kernel has a non-trivial centre. As a consequence of Proposition~\ref{prop:TAlmostSimple}, it suffices to show that:

\begin{claim}
The centraliser of $T_{n,m}^s$ in $T_{n,m}$ is trivial. 
\end{claim}

\noindent
Let $g \in T_{n,m}$ be an element centralising $T_{n,m}^s$. Fix an $n$-adic number $x \in \mathbb{R}/m\mathbb{Z}$. We can find an element $f$ in $T_{n,m}^s$ whose support in the circle $\mathbb{R}/m \mathbb{Z}$ is an interval with $x$ as an endpoint (e.g.\ take an arbitrary element of $T_{n,m}^s$ whose support is an interval and conjugate it by an element of $T_{n,m}^0$ in order to send this interval to an interval having $x$ as an endpoint). Because $g$ commutes with $f$, it has to stabilise the support of $f$, hence $g(x)=x$. We conclude that $g$ fixes every $n$-adic number, which implies that it must be the identity.
\end{proof}

\noindent
Next, we oberve that normal subgroups of $\mathrm{br}T_{n,m}$ that satisfies ($\ast$) are contained in $B_\infty$.

\begin{lemma}\label{lem:NotAst}
If a normal subgroup $N \lhd \mathrm{br}T_{n,m}$ is not contained in $B_\infty$, then $\mathrm{br}T_{n,m}$ surjects onto a virtually abelian group with a kernel that has a non-trivial centre.
\end{lemma}

\begin{proof}
Let $\pi$ denote the projection $\mathrm{br}T_{n,m} \twoheadrightarrow T_{n,m}$. According to Proposition~\ref{prop:TAlmostSimple}, the normal subgroup $\pi(N)$ in $T_{n,m}$ either is trivial or it contains $T_{n,m}^s$. In the former case, $N$ is contained in $B_\infty$ (which coincides with the kernel of $\pi$), which is forbidden by assumption. So $\pi(N)$ must contain $T_{n,m}^s$. Because $T_{n,m}^s$ is the commutator subgroup of the finite-index subgroup $T_{n,m}^0$ of $T_{n,m}$, this implies that $\mathrm{br}T_{n,m}/N$ is virtually abelian. It remains to verify that $B_\infty/(B_\infty \cap N)$ has a non-trivial centre.

\medskip \noindent
Because $\pi(N)$ contains $T_{n,m}^s$, we can find an element $g \in N$ such that the action of $g$ on the space of ends of $\mathscr{S}(A_{n,m})$ has an attracting point. (Notice that the action of $\mathrm{br}T_{n,m}$ on the space of ends of $\mathscr{S}(A_{n,m})$ is $\pi$-equivariantly equivalent to the action of $T_{n,m}$ on $\partial A_{n,m}$.) Consequently, there exists an infinite connected union of polygons $P \subset \mathscr{S}(A_{n,m})$ such that the $g^kP$ are pairwise disjoint for $k \geq 1$. Now, fix an arbitrary braid $\beta \in B_\infty \backslash N$. Up to conjugating by an element of $B_\infty$, we can assume that the support of $\beta$ is contained in $P$. Notice that, because $N$ is a normal subgroup, $\beta$ still does not belong to $N$. We claim that (the image of) $\beta$ is central in $B_\infty/ (B_\infty \cap N)$. 

\medskip \noindent
Indeed, if $\alpha$ is another braid, then there exists some $k \geq 1$ such that $\alpha$ and $g^k\beta g^{-k}$ have disjoint supports, and consequently commute in $B_\infty$. But $\beta$ and $g^k\beta g^{-k}$ coincide modulo $N$, so the images of $\alpha$ and $\beta$ in $B_\infty/ (B_\infty \cap N)$ must commute.
\end{proof}

\begin{proof}[Proof of Theorem~\ref{thm:BraidSubCharacteristic}.]
We know from Lemma~\ref{lem:NotAst} that every subgroup satisfying ($\ast$) is contained in $B_\infty$, and we know from Lemma~\ref{lem:BhasAst} that $B_\infty$ satisfies ($\ast$). Thus, our theorem follows. 
\end{proof}

\noindent
Now, we deduce by standard arguments a partial solution to the isomorphism problem among Thompson groups.

\begin{prop}\label{prop:IsoThompson}
	Let $n,m,r,s \geq 2$ be integers. If $T_{n,m}$ and $T_{r,s}$ are isomorphic, then $n=r$.
\end{prop}

\noindent
We are grateful to Jim Belk for having explained to us that the proposition is a rather straightforward consequence of Rubin's theorem.

\begin{proof}[Proof of Proposition~\ref{prop:IsoThompson}.]
We think of $T_{n,m}$ and $T_{r,s}$ as acting by piecewise linear homeomorphisms on $\mathbb{R}/n \mathbb{Z}$ and $\mathbb{R}/r \mathbb{Z}$ respectively. As a consequence of Rubin's theorem \cite[Corollary~3.5]{MR0988881}, if $T_{n,m}$ and $T_{r,s}$ are isomorphic, then there must exist a homeomorphism $\mathbb{R}/n \mathbb{Z} \to \mathbb{R}/r \mathbb{Z}$ that is equivariant with respect to the actions of $T_{n,m}$ and $T_{r,s}$. Claim~\ref{claim:Germs} below justifies that such a homeomorphism necessarily sends $n$-adic numbers to $r$-adic numbers, which implies that the number of $T_{n,m}$-orbits of pairs of distincts $n$-adic numbers in $\mathbb{R}/n \mathbb{Z}$ must equal the number of $T_{r,s}$-orbits of pairs of distinct $r$-adic numbers in $\mathbb{R}/r \mathbb{Z}$. But we know from \cite[Theorem~A4.1]{MR3560537} (see also \cite[Proposition~1]{MR4161164} for a proof focused on the case we are interested in) that these numbers are respectively $n-1$ and $r-1$. Hence $n=r$, as desired. 

\begin{claim}\label{claim:Germs}
Let $p,q \geq 2$ be two integers. For every $x \in \mathbb{R}/p \mathbb{Z}$, the group of germs of $T_{p,q}$ at $x$ is isomorphic to $\mathbb{Z}^2$ if $x$ is $p$-adic, to $\mathbb{Z}$ if $x$ is rational but not $p$-adic, and trivial if $x$ is irrational.
\end{claim}

\noindent
Recall that, given a group $G$ acting on a topological space $X$ and a point $x \in X$, the \emph{group of germs at $x \in X$} is the quotient $\mathrm{stab}(x)/ \mathrm{rig}(x)$, where $\mathrm{rig}(x)$ is the normal subgroup of $\mathrm{stab}(x)$ given by the elements fixing pointwise some neighbourhood of $x$.

\medskip \noindent
Claim~\ref{claim:Germs} can be proved by using the morphism
$$\Theta : \left\{ \begin{array}{ccc} \mathrm{stab}(x) & \to & \mathbb{Z} \times \mathbb{Z} \\ g & \mapsto & \left( \log(g'(x^-))/\log(p), \log(g'(x^+))/\log(p) \right) \end{array} \right.,$$
which gives the left- and right-derivates at $x$. Notice that the kernel of $\Theta$ coincides with $\mathrm{rig}(x)$, so the group of germs we are looking for is the image of $\Theta$. If $x$ is a $p$-adic number, we can construct elements of $T_{p,q}$ fixing $x$ and having left- and right-derivatives equal to arbitrary powers of $p$. In this case, $\Theta$ is surjective. If $x$ is irrational, then the identity is the only element of $T_{p,q}$ fixing $x$, since locally every element of $T_{p,q}$ is an affine map with rational coefficients. So the image of $\Theta$ is trivial in this case. If $x$ is rational but not $p$-adic, then the left- and right-derivatives of an element of $T_{p,q}$ fixing $x$ must be equal, but they can take as a common value an arbitrary power of $p$. In other words, the image of $\Theta$ is the infinite cyclic subgroup $\{(a,a) \mid a \in \mathbb{Z}\}$ of $\mathbb{Z}^2$. 
\end{proof}

\begin{proof}[Proof of Theorem~\ref{thm:BigIntro}.]
	As a consequence of Theorem~\ref{thm:BraidSubCharacteristic}, an isomorphism $\mathrm{br}T_{n,m} \to \mathrm{br}T_{r,s}$ induces an isomorphism $T_{n,m} \to T_{r,s}$, which implies that $r=n$ according to Proposition~\ref{prop:IsoThompson}. But we know from Theorem~\ref{thm_abelianise} that the abelianisation of $\mathrm{br}T_{n,m}$ (resp. $\mathrm{br}T_{r,s}$) has order $m |m-n+1|$ (resp. $s |s-r+1|$). It follows from Claim~\ref{claim:arithm} that if $m\neq s$ then there are three families to distinguish:
	\begin{itemize}
		\item $\mathrm{br}T_{d(u^2+v^2)+1,dv(u+v)}$ and $\mathrm{br}T_{d(u^2+v^2)+1,du(u+v)}$ where $u > v$. By \cite{GLU_finiteness}, the first group contains an element of order $\ell$ if and only if $\ell$ divides $dv(u+v)$ or $du(u-v)$; and the second group contains an element of order $\ell$ if and only if $\ell$ divides $du(u+v)$ or $dv(u-v)$. Since $du(u+v)$ is larger than $dv(u+v)$ and $du(u-v)$, the two groups cannot be isomorphic because only the second one contains an element of order $du(u+v)$. 
		\item $\mathrm{br}T_{d(u^2+v^2)+1,du(u-v)}$ and $\mathrm{br}T_{d(u^2+v^2)+1,du(u+v)}$ where $u > v$. The first group contains an element of order $\ell$ if and only if $\ell$ divides $du(u-v)$ or $dv(u+v)$. The second group contains an element of order $\ell$ if and only if $\ell$ divides $du(u+v)$ or $dv(u-v)$. Since $du(u+v)$ is larger than $du(u-v)$ and $dv(u+v)$, it follows that the two groups cannot be isomorphic because only the second one contains an element of order $du(u+v)$.
		\item $\mathrm{br}T_{n,m}$ and $\mathrm{br}T_{n,n-1-m}$ where $2 \leq m \leq (n-1)/2$ is the only possibility remaining.
	\end{itemize}
	
	\begin{claim}\label{claim:arithm}
		If $x|x-k| = y|y-k|$ where $0\leq x < y$, and $k\geq 1$ then
		\begin{itemize}
			\item $0 \leq x \leq k/2$ and $y=k-x$;
			\item or $0 \leq x \leq k/2$ and $x=dv(u+v)$, $y=du(u+v)$ where $u > v$ and $d\in\Z_{\geq0}$ are such that $k=d(u^2+v^2)$;
			\item or $k/2 \leq x \leq k$ and $x=du(u-v)$, $y=du(u+v)$ where $u > v$ and $d\in\Z_{\geq0}$ are such that $k=d(u^2+v^2)$. 
		\end{itemize}
	\end{claim}
	
	\noindent
	The map $z \mapsto z|z-k|$ increases on $[0,k/2]$, decreases on $[k/2,k]$, and increases again on $[k,+\infty)$, so either $0 \leq x \leq k/2$ and $k/2 \leq y \leq k$ or $0 \leq x \leq k$ and $y \geq k$.
	
	\medskip \noindent
	In the first case, we have $x(x-k)=y(y-k)$, which can be rewritten as $(x^2-y^2) - k(x-y)=0$. Dividing by $x-y$, we get $y=k-x$ as desired.
	
	\medskip \noindent
	In the second case, we have $-x(x-k)=y(y-k)$, which can be rewritten as $(x-k/2)^2+(y-k/2)^2 = 2 (k/2)^2$, or equivalently $(2x-k)^2+(2y+k)^2= 2k^2$. The Diophantine equation $X^2+Y^2=2Z^2$ is classical and the solutions are known. It follows that there exist $u,v$ with $u\geq v$ such that
	$$\left\{ \begin{array}{l} k=d(u^2+v^2 )\\ 2x-k = d(u^2-v^2-2uv )\\ 2y-k= d(u^2-v^2+2uv) \end{array} \right. \text{ if } k/2 \leq x \leq k$$
	for some constant $d\in\Z_{\geq 0}$.
	We obtain
	$$\left\{ \begin{array}{l} k=d(u^2+v^2) \\ k-2x = d(u^2-v^2-2uv) \\ 2y-k= d(u^2-v^2+2uv) \end{array} \right. \text{ if } 0 \leq x \leq k/2.$$
	Moreover, if $u=v$ then, when $k/2\leq x\leq k$, $2x-k=d(-2u^2)$ that implies that $u=0$ and so $k=0$, which contradicts the assumption on $k$. When $0\leq x\leq k/2$, a similar argument implies the desired conclusion.
\end{proof}

	\addcontentsline{toc}{section}{References}
	
	\bibliographystyle{alpha}
	{\footnotesize\bibliography{bibliography}}

\begin{thebibliography}{ABKL24}

\bibitem[ABKL24]{SurfaceHoughton}
Javier Aramayona, Kai-Uwe Bux, Heejoung Kim, and Christopher~J. Leininger.
\newblock Surface {H}oughton groups.
\newblock {\em Math. Ann.}, 389(4):4301--4318, 2024.

\bibitem[AF21]{Aramoyana_Funar_asymptotic_MCG}
Javier Aramayona and Louis Funar.
\newblock Asymptotic mapping class groups of closed surfaces punctured along
  {C}antor sets.
\newblock {\em Mosc. Math. J.}, 21(1):1--29, 2021.

\bibitem[Bro84]{Brown_presentation}
K.~S. Brown.
\newblock Presentations for groups acting on simply-connected complexes.
\newblock {\em J. Pure Appl. Algebra}, 32(1):1--10, 1984.

\bibitem[Bro87]{Brown_finiteness_properties}
Kenneth~S. Brown.
\newblock Finiteness properties of groups.
\newblock In {\em Proceedings of the {N}orthwestern conference on cohomology of
  groups ({E}vanston, {I}ll., 1985)}, volume~44, pages 45--75, 1987.

\bibitem[BS16]{MR3560537}
Robert Bieri and Ralph Strebel.
\newblock {\em On groups of {PL}-homeomorphisms of the real line}, volume 215
  of {\em Mathematical Surveys and Monographs}.
\newblock American Mathematical Society, Providence, RI, 2016.

\bibitem[Deg00]{Degenhardt}
F.~Degenhardt.
\newblock {\em Endlichkeitseigeinschaften gewisser {G}ruppen von {Z}\"{o}pfen
  unendlicher {O}rdnung}.
\newblock PhD thesis, Frankfurt 2000.

\bibitem[FK04]{Funar_Kapoudjian_UniversalMCG}
L.~Funar and C.~Kapoudjian.
\newblock On a universal mapping class group of genus zero.
\newblock {\em Geom. Funct. Anal.}, 14(5):965--1012, 2004.

\bibitem[FK08]{Funar-Kapoudjian_brT_finitely_presented}
L.~Funar and C.~Kapoudjian.
\newblock The braided {P}tolemy-{T}hompson group is finitely presented.
\newblock {\em Geom. Topol.}, 12(1):475--530, 2008.

\bibitem[Fun07]{Funar_Houghton}
L.~Funar.
\newblock Braided {H}oughton groups as mapping class groups.
\newblock {\em An. \c{S}tiin\c{t}. Univ. Al. I. Cuza Ia\c{s}i. Mat. (N.S.)},
  53(2):229--240, 2007.

\bibitem[GLU22]{GLU_finiteness}
Anthony Genevois, Anne Lonjou, and Christian Urech.
\newblock Asymptotically rigid mapping class groups, {I}: {F}initeness
  properties of braided {T}hompson's and {H}oughton's groups.
\newblock {\em Geom. Topol.}, 26(3):1385--1434, 2022.

\bibitem[GLU25]{GLU_Chambord}
Anthony Genevois, Anne Lonjou, and Christian Urech.
\newblock Asymptotically rigid mapping class groups {II}: {S}trand diagrams and
  non-positive curvature.
\newblock {\em Trans. Amer. Math. Soc.}, 378(8):5355--5402, 2025.

\bibitem[HBAL20]{MR4161164}
Hajer Hmili Ben~Ammar and Isabelle Liousse.
\newblock Nombre de classes de conjugaison d'\'{e}l\'{e}ments d'ordre fini dans
  les groupes de {B}rown-{T}hompson.
\newblock {\em Bull. Soc. Math. France}, 148(3):399--409, 2020.

\bibitem[Rub89]{MR0988881}
Matatyahu Rubin.
\newblock On the reconstruction of topological spaces from their groups of
  homeomorphisms.
\newblock {\em Trans. Amer. Math. Soc.}, 312(2):487--538, 1989.

\bibitem[Ser93]{Sergiescu_presentation_tresses}
Vlad Sergiescu.
\newblock Graphes planaires et pr\'{e}sentations des groupes de tresses.
\newblock {\em Math. Z.}, 214(3):477--490, 1993.

\end{thebibliography}
	\Address
\end{document}